\documentclass[11pt]{amsart}
\usepackage{amscd}
\usepackage{amsmath}
\usepackage{amsxtra}
\usepackage{amsfonts}
\usepackage{amssymb}

 \newenvironment{myproof}[1][\proofname]{%
  \proof[\bf #1 ]%
}{\endproof}

\newtheorem{theorem}{Theorem}[section]
\newtheorem{corollary}[theorem]{Corollary}
\newtheorem{lemma}[theorem]{Lemma}
\newtheorem{proposition}[theorem]{Proposition}

\newtheorem{conjecture}[theorem]{Conjecture}
\theoremstyle{definition}
\newtheorem{definition}[theorem]{Definition}
\newtheorem{remark}[theorem]{Remark}

\newtheorem{example}[theorem]{Example}
\theoremstyle{remark}

\renewcommand{\theclaim}{\textup{\theclaim}}

\newtheorem*{acknowledgements}{Acknowledgements}

\numberwithin{equation}{section}

\def\openone%{\hbox{\upshape \small1\kern-3.3pt\normalsize1}}

{\mathchoice

{\hbox{\upshape \small1\kern-3.3pt\normalsize1}}

{\hbox{\upshape \small1\kern-3.3pt\normalsize1}}

{\hbox{\upshape \tiny1\kern-2.3pt\SMALL1}}

{\hbox{\upshape \Tiny1\kern-2pt\tiny1}}}

\makeatletter

\newbox\ipbox

\newcommand{\diracb}[1]{\left\langle #1\mathrel{\mathchoice

{\setbox\ipbox=\hbox{$\displaystyle \left\langle\mathstrut
#1\right.$}

\vrule height\ht\ipbox width0.25pt depth\dp\ipbox}

{\setbox\ipbox=\hbox{$\textstyle \left\langle\mathstrut
#1\right.$}

\vrule height\ht\ipbox width0.25pt depth\dp\ipbox}

{\setbox\ipbox=\hbox{$\scriptstyle \left\langle\mathstrut
#1\right.$}

\vrule height\ht\ipbox width0.25pt depth\dp\ipbox}

{\setbox\ipbox=\hbox{$\scriptscriptstyle \left\langle\mathstrut
#1\right.$}

\vrule height\ht\ipbox width0.25pt depth\dp\ipbox}

}\right. }

\newcommand{\dirack}[1]{\left. \mathrel{\mathchoice

{\setbox\ipbox=\hbox{$\displaystyle \left.\mathstrut
#1\right\rangle$}

\vrule height\ht\ipbox width0.25pt depth\dp\ipbox}

{\setbox\ipbox=\hbox{$\textstyle \left.\mathstrut
#1\right\rangle$}

\vrule height\ht\ipbox width0.25pt depth\dp\ipbox}

{\setbox\ipbox=\hbox{$\scriptstyle \left.\mathstrut
#1\right\rangle$}

\vrule height\ht\ipbox width0.25pt depth\dp\ipbox}

{\setbox\ipbox=\hbox{$\scriptscriptstyle \left.\mathstrut
#1\right\rangle$}

\vrule height\ht\ipbox width0.25pt depth\dp\ipbox}

} #1\right\rangle}

\newcommand{\beq}{\begin{equation}}

\newcommand{\eeq}{\end{equation}}

\newcommand{\cj}[1]{\overline{#1}}

\newcommand{\bz}{\mathbb{Z}}

\newcommand{\br}{\mathbb{R}}
\newcommand{\bc}{\mathbb{C}}
\newcommand{\bt}{\mathbb{T}}
\newcommand{\bn}{\mathbb{N}}

\def\blfootnote{\xdef\@thefnmark{}\@footnotetext}
\newcommand{\uln}[1]{\underline{#1}}
\renewcommand{\mod}{\operatorname{mod}}

\input xy
\xyoption{all}
\usepackage{amssymb}

\def\-{^{-1}}

\begin{document}

\title[Tiling properties of spectra of measures]{Tiling properties of spectra of measures}
\author{Dorin Ervin Dutkay}

\address{[Dorin Ervin Dutkay] University of Central Florida\\
	Department of Mathematics\\
	4000 Central Florida Blvd.\\
	P.O. Box 161364\\
	Orlando, FL 32816-1364\\
U.S.A.\\} \email{Dorin.Dutkay@ucf.edu}

\author{John Haussermann}

\address{[John Haussermann] University of Central Florida\\
	Department of Mathematics\\
	4000 Central Florida Blvd.\\
	P.O. Box 161364\\
	Orlando, FL 32816-1364\\
U.S.A.\\} \email{jhaussermann@knights.ucf.edu}

\thanks{} 
\subjclass[2000]{42A16,05B45,15B34}
\keywords{Spectrum, tile, Hadamard matrix, fractal, affine iterated function system, Cantor set, Fuglede conjecture, Coven-Meyerowitz conjecture}

\begin{abstract}
 We investigate tiling properties of spectra of measures, i.e., sets $\Lambda$ in $\br$ such that $\{e^{2\pi i \lambda x}: \lambda\in\Lambda\}$ forms an orthogonal basis in $L^2(\mu)$, where $\mu$ is some finite Borel measure on $\br$. Such measures include Lebesgue measure on bounded Borel subsets, finite atomic measures and some fractal Hausdorff measures. We show that various classes of such spectra of measures have translational tiling properties. This lead to some surprizing tiling properties for spectra of fractal measures, the existence of complementing sets and spectra for finite sets with the Coven-Meyerowitz property, the existence of complementing Hadamard pairs in the case of Hadamard pairs of size 2,3,4 or 5. In the context of the Fuglede conjecture, we prove that any spectral set is a tile, if the period of the spectrum is 2,3,4 or 5. 
\end{abstract}
\maketitle \tableofcontents

\section{Introduction}
\begin{definition}\label{def1.0}
For $\lambda\in \br$ we denote by $e_\lambda(x):=e^{2\pi i\lambda\cdot x}$. We say that a finite Borel measure $\mu$ on $\br$ is {\it spectral} if 
there exists a set $\Lambda$ such that the family of exponential functions $E(\Lambda):=\{e_\lambda : \lambda\in\Lambda\}$ is an orthogonal basis for $L^2(\mu)$. We call $\Lambda$ a {\it spectrum} for $\mu$. If $E(\Lambda)$ is an orthogonal set then we say that $\Lambda$ is {\it orthogonal}.

We say that a bounded Borel subset $\Omega$ of $\br$ is {\it spectral} if the restriction of the Lebesgue measure to $\Omega$ is a spectral measure.

 We say that a finite subset $A$ of $\br$ is {\it spectral} if the counting measure on $A$ is a spectral measure. 
\end{definition}

Spectral sets have been introduced in relation to the Fuglede conjecture \cite{Fug74}:
\begin{conjecture}
A bounded Borel subset $\Omega$ of $\br$ is spectral if and only if it tiles $\br$ by translations, i.e., there exists a set $\mathcal T$ in $\br$ such that $\{\Omega+t : t\in\mathcal T\}$ is a partition of $\br$ (up to Lebesgue measure zero). 
\end{conjecture}

The conjecture can be formulated in any dimension, but it is known to be false in both directions for dimensions 3 or higher \cite{Tao04,FaMaMo06}. In dimensions 1 and 2, as far as we know at the moment of writing this paper, the conjecture is open in both directions.

In \cite{JoPe98}, Jorgensen and Pedersen have constructed a new example of a spectral measure, a fractal one. Their construction is based on a scale 4 Cantor set, where the first and and third intervals are kept and the other two are discarded. The appropriate measure for this set is the Hausdorff measure $\mu_4$ of dimension $\frac{\ln 2}{\ln 4}=\frac12$. They proved that this measure is spectral with spectrum 
\begin{equation}
\Lambda:=\left\{\sum_{k=0}^n 4^kl_k: l_k\in\{0,1\}, n\in\bn\right\}.
\label{eqspmu4}
\end{equation}

Many other examples of fractal measures have been constructed since, see e.g. \cite{MR1785282,LaWa02,DJ06,DJ07d}, and many other spectra can be constructed for the same measure, see e.g., \cite{DHS09}. Among other things, we will show that the spectrum $\Lambda$ in \eqref{eqspmu4} tiles $\bz$ by translations.

A large class of examples of spectral measures is based on affine iterated function systems. 

\begin{definition}\label{def1.1}
Let $R$ be an integer $R\geq 2$. We call $R$ the {\it scaling factor}. Let $B\subset\bz$, $0\in\bz$, $N:=\#B$. We define the affine iterated function system
$$\tau_b(x)=R^{-1}(x+b),\quad(x\in\br,b\in B).$$
By \cite{Hut81} there exist a unique compact set $X_B$ called {\it the attractor} of the IFS $\{\tau_b\}$, such that 
$$X_B=\cup_{b\in B}\tau_b(X_B).$$
The set $X_B$ can be described using the base $R$ representation of real numbers, with digits in $B$:
$$X_B=\left\{\sum_{k=1}^\infty R^{-k}b_k : b_k\in B\right\}.$$
Also by \cite{Hut81}, there exists a unique Borel probability measure $\mu_B$ on $\br$ that satisfies the invariance equation 
\begin{equation}
\mu_B(E)=\frac{1}{N}\sum_{b\in B}\mu_B(\tau_b^{-1}E)\mbox{ for all Borel subsets $E$ of $\br$}.
\label{eq1.1.1}
\end{equation}
Equivalently, for all continuous compactly supported functions $f$:
\begin{equation}
\int f\,d\mu_B=\frac1N\sum_{b\in B}\int f\circ\tau_b\,d\mu_B.
\label{eq1.1.2}
\end{equation}
The measure $\mu_B$ is called {\it the invariant measure} of the IFS $\{\tau_b\}$. 
In addition the measure $\mu_B$ is supported on the attractor $X_B$.
\end{definition} 

\begin{definition}\label{def1.1i}
Let $L\subset\bz$, $0\in L$. We say that $(B,L)$ is a {\it Hadamard pair with scaling factor $R$} if $\#L=\#B=N$ and the matrix 
$$\frac{1}{\sqrt N}\left(e^{2\pi i R^{-1}b\cdot l}\right)_{b\in B,l\in L}$$
is unitary. We call this matrix {\it the matrix associated with $(B,L)$}.

We define the function 
\begin{equation}
m_B(x)=\frac{1}{N}\sum_{b\in B}e^{2\pi ib\cdot x},\quad (x\in\br)
\label{eq1.1i.1}
\end{equation}

Given a Hadamard pair $(B,L)$ we say that a finite set of points $\{x_0,\dots,x_{r-1}\}$ in $\br$ is a {\it cycle} for $L$ if there exist $l_0,\dots,l_{r-1}$ in $L$ such that 
$$\frac{x_0+l_0}{R}=x_1,\dots,\frac{x_{r-2}+l_{r-2}}{R}=x_{r-1},\frac{x_{r-1}+l_{r-1}}{R}=x_0.$$
We call $l_0,\dots,l_{r-1}$ {\it the digits of this cycle.}
We say that this cycle is {\it extreme} for $(B,L)$ if  
$$|m_B(x_k)|=1\mbox{ for all }k\in\{0,\dots,r-1\}.$$
The points $\{x_i\}$ are called {\it (extreme) cycle points}.
\end{definition}

When $(B,L)$ is a Hadamard pair with scaling factor $R$, then the measure $\mu_B$ is always spectral and a spectrum can be constructed using digits in $L$ and extreme cycles. 

\begin{theorem}\label{th1.2}\cite{DJ06}
If $(B,L)$ is a Hadamard pair then $\mu_B$ is a spectral measure with spectrum $\Lambda(L)$ where $\Lambda$ is the smallest set which contains $-C$ for all cycles $C$ for $L$ which are extreme for $(B,L)$, and with the property that $R\Lambda(L)+L\subset \Lambda(L)$. 
\end{theorem}

This spectrum can be described in terms of base $R$ representations of integers using only digits in $L$. 

\begin{definition}\label{def1.2}
Let $L$ be a set of integers. We say that an integer $x$ can be represented in base $R$ using digits in $L$ if there exist integers $x_0,x_1,\dots$, with $x_0=x$ and digits $l_0,l_1,\dots$ in $L$ such that 
$$x_k=Rx_{k+1}+l_k\mbox{ for all $k\geq0$}.$$
We call $l_0l_1\dots$ a {\it representation} of $x$ in base $R$. 
\end{definition}

\begin{proposition}\label{pr0.1.6}
Let $(B,L)$ be a Hadamard pair. Assume in addition that all extreme cycles for $(B,L)$ are contained in $\bz$. Then the spectrum $\Lambda(L)$ defined in Theorem \ref{th1.2} is the set of integers which can be represented in base $R$ using digits in $L$.
\end{proposition}

Next we turn our attention to finite spectral subsets of $\bz$. The variant of the Fuglede conjecture for such sets is that a finite subset $A$ of $\bz$ is spectral if and only if it tiles $\bz$ by translations. In \cite{CoMe99}, Coven and Meyerowitz proposed a characterization of sets that tile integers by translations, in terms of cyclotomic polynomials. 
\begin{definition}\label{def2.1.1}
Let $A$ be a finite multiset of nonnegative integers, by multiset we mean that some elements $a\in A$ might be counted with multiplicity $m_a$. We define the polynomial corresponding to $A$ by 
$$A(x)=\sum_{a\in A}m_ax^a.$$

For $s\in\bn$, we denote by $\Phi_s(x)$ the $s$-th cyclotomic polynomial. We denote by $S_A$ the set of all prime powers such that the $s$-th cyclotomic polynomial divides $A(x)$.

We say that the set $A$ (without any multiplicity) {\it satisfies the Coven-Meyerowitz property (or shortly, $A$ has the CM-property) }if the following two conditions are satsisfied:
\begin{enumerate}
	\item[(T1)] $A(1)=\prod_{s\in S_A}\Phi_s(1).$
	\item[(T2)] If $s_1,\dots,s_m\in S_A$ are powers of distinct primes then $\Phi_{s_1\dots s_m}(x)$ divides $A(x)$.
\end{enumerate} 
\end{definition}

Coven and Meyerowitz proved in \cite{CoMe99} that a set with the CM-property tiles $\bz$ by translations and they conjectured that the reverse is also true, and proved the conjecture in some special cases (when the size of the set has at most two prime factors). \L aba proved in \cite{Lab02} that the CM-property also implies that the set is spectral. Combining these results we show that the tiling sets and spectra fit together nicely in a {\it complementary pair}. We are also interested in the extreme cycles due to their importance for the spectra of fractal measures.

\begin{definition}\label{def2.0}
Let $A,A'$ be two subsets of $\br$. We say that $A$ and $A'$ {\it have disjoint differences} if $(A-A)\cap(A'-A')=\{0\}$. In this case we denote by 
$A\oplus A'=\{a+a' : a\in A, a'\in A'\}$; we use the sign $\oplus$ to indicate that the sets have disjoint differences; equivalently, for any $x\in A+A'$ there exist unique $a\in A$ and $a'\in A'$ such that $x=a+a'$; equivalently, the sets $A+a'$, $a'\in A'$ are disjoint. 
\end{definition}

\begin{definition}\label{def2.1}
Let $R\in\bz$, $R\geq 2$. Let $(B,L)$ and $(B',L')$ be two Hadamard pairs with scaling factor $R$, $\#B=N$, $\#B'=N'$, not necessarily equal. We say that the two Hadamard pairs are {\it complementary} if the following conditions are satisfied: 
\begin{enumerate}
	\item $B\oplus B'$ and $L\oplus L'$ are complete sets of representatives $\mod R$.
	\item The extreme cycles for $(B,L)$ and the extreme cycles for $(B',L')$ are contained in $\bz$.
	\item The greatest common divisor of the points in $B\oplus B'$ is 1.
\end{enumerate}
\end{definition}

\begin{theorem}\label{th0.2.1.2}
Let $B$ a finite set of nonnegative integers with $\gcd(B)=1$ and which satsifies the Coven-Meyerowitz property. Let $R$ be the lowest common multiple of the elements in $S_B$. Then there exist finite sets $B',L,L'$ of nonnegatve integers such that 
\begin{enumerate}
	\item $(B,L)$ and $(B',L')$ are complementary Hadamard pairs (relative to the number $R$).
	\item $B'$ satisfies the Coven-Meyerowitz property.
\end{enumerate}
\end{theorem}

Once we have two complementary Hadamard pairs $(B,L)$, $(B',L')$ with scaling factor $R$, we can construct the two fractal measures $\mu_B$ and $\mu_{B'}$ with spectra $\Lambda(L)$ and $\Lambda(L')$ respectively. The next theorem shows that the convolution of the two measures $\mu_B$ and $\mu_{B'}$ is the Lebesgue measure on a tile of $\br$, it is also the invariant measure $\mu_{B\oplus B'}$ for the affine IFS associated to scaling by $R$ and digits $B\oplus B'$. The two spectra always have disjoint differences and moreover, under some restrictions on the encodings of the extreme cycles for $(B\oplus B',L\oplus L')$, $(B,L)$ and $(B',L')$, the two sets complement each other, in the sense that $\Lambda(L)$ tiles $\bz$ with $\Lambda(L')$.

\begin{definition}\label{def2.2}
Let $(B,L)$ and $(B',L')$ be complementary Hadamard pairs with scaling factor $R$. We define the maps $p:L\oplus L'\rightarrow L$ and $p':L\oplus L'\rightarrow L'$ by 
$$p(l+l')=l,\quad p'(l+l')=l'\mbox{ for all }l\in L ,l'\in L'.$$
For a sequence $a_0a_1\dots$ of digits in $L\oplus L'$ we define 
$$p(a_0a_1\dots)=p(a_0)p(a_1)\dots,\quad p'(a_0a_1\dots)=p'(a_0)p'(a_1)\dots.$$
\end{definition}

\begin{theorem}\label{th2.3}
Let $(B,L)$ and $(B',L')$ be complementary Hadamard pairs with scaling factor $R$. Let $\Lambda(L)$ be the set of integers that can be represented in base $R$ using digits from $L$, and similarly for $\Lambda(L')$.
\begin{enumerate}
	\item The measure $\mu_{B\oplus B'}$ is the Lebesgue measure on the attractor $X_{B\oplus B'}$ and has spectrum $\bz$. Moreover $X_{B\oplus B'}$ is translation congruent to $[0,1]$, i.e., there exists a measurable partition $\{A_n\}_{n\in\bz}$ of $[0,1]$ such that $\{A_n+n\}_{n\in\bz}$ is a partition of $X_{B\oplus B'}$. 
	\item The measure $\mu_{B\oplus B'}$ is the convolution of the measures $\mu_B$ and $\mu_{B'}$. 
	\item The set $\Lambda(L)$ is a spectrum for $\mu_B$ and the set $\Lambda(L')$ is a spectrum for $\mu_{B'}$.
	\item The sets $\Lambda(L)$ and $\Lambda(L')$ have disjoint differences. 	
	\item The set $\Lambda(L)\oplus\Lambda(L')=\bz$ if and only if for any digits $a_0\dots a_{r-1}$ of an extreme cycle for $(B\oplus B',L\oplus L')$, the sequence $p(a_0\dots a_{r-1})$ consists of the digits of an extreme cycle for $(B,L)$ and the sequence $p'(a_0\dots a_{r-1})$ consists of the digits of an extreme cycle for $(B',L')$. 
	The equality $\Lambda(L)\oplus\Lambda(L')=\bz$ means that $\Lambda(L)$ tiles $\bz$ by $\Lambda(L')$.
\end{enumerate}

\end{theorem}

Next, we focus on sets $B$ of small size: 2,3,4,5 and investigate when such a set is spectral and when a Hadamard pair with scaling factor $R$ can be complemented. We base our results on the classification of Hadamard matrices of size 2,3,4,5. For size $\#B=2,3,4$ this is fairly simple, see \cite{TaZy06}. For size 5, the problem becomes more complicated but it was solved by Haagerup \cite{Haa97}.

\begin{definition}
A $N\times N$ matrix $H$ is called a {\it Hadamard matrix} if it is unitary and all its entries have the same absolute value $\frac{1}{\sqrt N}$. Two Hadamard matrices $H$, $H$' are said to be {\it equivalent} if one can be obtained from the other after permutations of row and columns and multiplication of rows and columns by complex numbers of absolute value 1; formally: there exist permutation $\pi$ and $\rho$ of the set $\{1,\dots,N\}$ and complex numbers $c_1,\dots,c_N,d_1,\dots,d_N$ on the unit circle $\bt=\{z: |z|=1\}$ such that 
$$H_{ij}'=c_id_jH_{\pi(i)\rho(j)},\quad(i,j\in\{1,\dots,N\}).$$
The matrix of the Fourier transform on $\bz_N$, $\frac{1}{\sqrt{N}}(e^{2\pi i\frac{jk}{N}})_{j,k=0}^{N-1}$ is called the {\it standard Hadamard matrix}.

\end{definition}

\begin{theorem}\label{th1.15}(See \cite{TaZy06,Haa97})
Let $N=2,3$ or $5$. Any Hadamard matrix of size $N$ is equivalent to the standard Hadamard matrix. If $N=4$, any $4\times 4$ Hadamard matrix is equivalent to one of the following form:
\begin{equation}\frac12
\begin{pmatrix}
1&1&1&1\\
1&-1&\rho&-\rho\\
1&-1&-\rho&\rho\\
1&1&-1&-1
\end{pmatrix}
\label{eqmat4}
\end{equation}
for some $\rho\in\bt$.
\end{theorem}

As far as we know, there is no classification for Hadamard matrices of size 6 or higher. Beauchamp and Nicoar\u a gave a classification of {\it self-adjoint} $6\times 6$ Hadamard matrices in \cite{MR2398121}.

A Hadamard matrix is said to be in de-phased form if its first row and column contain only the number $1$.

\begin{corollary} \label{perm}
Let $N=2$, $3$, $4$, or $5$. Any two Hadamard matrices $A$ and $B$ of size $N$ in de-phased form which are equivalent are also equivalent via permutations only, that is, there are permutation matrices $P_1$ and $P_2$ such that $A=P_1 B P_2$.

\end{corollary}
\begin{definition}\label{def0.3.1}
Let $B$ be a finite spectral subset of $\br$ with spectrum $\Lambda$, $\#B=\#\Lambda=:N$. The matrix 
$$\frac{1}{\sqrt{N}}(e^{2\pi ib\cdot\lambda})_{b\in B,\lambda\in\Lambda}$$
is a Hadamard matrix and we called it {\it the Hadamard matrix associated to $B$ and $\Lambda$}.
\end{definition}

This enables us to describe the spectral sets of size $2,3,4,5$.

\begin{theorem} \label{standard}
Let $B \subset \bz$ have $N$ elements and spectrum $\Lambda$. Assume $0$ is in $B$ and $\Lambda$. Suppose the Hadamard matrix associated to $(B,\Lambda)$ is equivalent to the standard $N$ by $N$ Hadamard matrix. Then $B$ has the form $B=d B_0$ where $d$ is an integer and $B$ is a complete set of residues modulo $N$ with $\gcd(B)=1$. In this case any such spectrum $\Lambda$ has the form $\Lambda = \frac{1}{R} f L_0$ where $f$ and $R$ are integers, $L_0$ is a complete set of residues modulo $N$ with greatest common divisor one, and $R=NS$ where $S$ divides $df$ and $\frac{df}{S}$ is mutually prime with $N$. The converse also holds.

\end{theorem}

\begin{corollary}\label{pr0.1}A set $B \subset \bz$ with $|B|=N=2$, $3$, or $5$, where $0 \in B$ is spectral if and only if $B= N^k B_0$ where $k$ is a positive integer and $B_0$ is a complete set of residues modulo $N$.
\end{corollary}

%
%\begin{theorem} \label{pr0.3}
%Let $B$ be a subset of $\bz$ with 4 elements, $0\in B$. Then $B$ is spectral if and only if it is of the form
%$$B=2^m\left\{ 0, 2^ac_1, c_2, c_2+2^ad\right\},$$
%with $m\in\bz$, $m\geq 0$, $a\in\bz$, $a\geq 1$, $c_1,c_2,d$ odd. In this case, a spectrum for $B$ is
%$$\Lambda:=\frac{1}{2^m}\left\{0,\frac12,\frac{1}{2^{a+1}},\frac12+\frac1{2^{a+1}}\right\}.$$
%\end{theorem}

We can also describe all possible Hadamard pairs of size 2,3,4,5.

%\begin{corollary}\label{thha235}
%Let $B,L$ be a Hadamard pair of integers of size $N$ (containing zero as their first element), with positive integer scaling factor $R$, where the matrix associated with $B,L$ is equivalent to the standard $N$ by $N$ Hadamard matrix. Then $B=h_0 N^f \{0 = b_0, b_1 , ... ,b_{N-1} \}$ and $L=h_1 N^g \{0 = l_0, l_1 , ... ,l_{N-1} \}$ for some non-negative integers $f$ and $g$, and $\{ b_i \}$ and $\{ l_i \}$ are complete sets of residues modulo $N$ which are respectively mutually prime. In addition, $R=\frac{ N^{f+g+1} h_0 h_1}{Z_2} $, where $N$ does not divide $h_0$ or $h_1$ and $Z_2$ and $N$ are relatively prime.
%\end{theorem}

\begin{theorem}\label{thha4}
Let $B$ be spectral with spectrum $\Lambda$ and size $N=4$. Assume $0$ is in both sets. Then there exists a set of integers $L$, containing $0$, and an integer scaling factor $R$ so that $\Lambda= \frac{1}{R} L$.

$(B,L)$ is a Hadamard pair (each containing $0$) of integers of size $N=4$, with scaling factor $R$, if and only if $R=2^{C+M+a+1} d$, $B=2^C \{0, 2^a c_1, c_2, c_2 + 2^a c_3\}$, and $L=2^M \{0, n_1, n_1 + 2^a n_2, 2^a n_3\}$, where $c_i$ and $n_i$ are all odd, $a$ is a positive integer, $C$ and $M$ are non-negative integers, and $d$ divides $c_1 n$, $c_3 n$, $n_2 c$, and $n_3 c$, where $c$ is the greatest common divisor of the $c_k$'s and similarly for $n$.
\end{theorem}

Using the classification of Hadamard matrices of small dimension we can also show that Hadamard pairs of size 2,3,4,5 can always be complemented. We can give a more general result:

\begin{theorem}\label{th0.4a}
Let $(B,L)$ be a Hadamard pair of integers of size $N$ (containing zero as their first element), with scaling factor an integer $R$, where the matrix associated with $(B,L)$ is equivalent to the $N\times N$ standard Hadamard matrix. Assume that all extreme cycles for $(B,L)$ are contained in $\bz$. Then $(B,L)$ has a complementary Hadamard pair of integers.
\end{theorem}

\begin{theorem}\label{th0.5a}
Let $(B,L)$ be a Hadamard pair of size $|B|=|L|=2,3,4$ or $5$, with scaling factor $R$, and assume all extreme cycles for $(B,L)$ are contained in $\bz$. Then $(B,L)$ has a complementary Hadamard pair.
\end{theorem}

The cases 2,3,5 follow imediately from Theorem \ref{th0.4a} since the Hadamard matrix associated to the pair $(B,L)$ has to be equivalent to the standard one. For size 4, the situation is different. 

A useful tool for our construction of Hadamard pairs is the following proposition, which is closely relation to Di\c{t}\u a's construction of Hadamard matrices (see e.g. \cite{TaZy06}):

\begin{proposition} \label{prHP}
 Let $(B,L)$ and $(F,G)$ be Hadamard pairs of integers with the same scaling factor $R$ and such that $b\cdot g$ is a multiple of $R$ for every $g\in G$ and $b\in B$. Then $(B\oplus F, L \oplus G)$ is a Hadamard pair with scaling factor $R$.
 \end{proposition}

Finally, we study spectral sets with Lebesgue measure as part of the original Fuglede conjecture. A wonderful result due to Iosevich and Kolountzakis \cite{IoKo12} states that the spectrum $\Lambda$ of a bounded spectral subset $\Omega$ of $\br$ has to be periodic. More precisely
\begin{theorem}(\cite{IoKo12}) Let $\Omega$ be a bounded Borel subset of $\br$ with Lebesgue measure $|\Omega|=1$. If $\Omega$ is spectral with spectrum $\Lambda$ then $\Lambda$ is periodic, i.e., there exists $p>0$ such that $\Lambda+p=\Lambda$; moreover the period $p$ is an integer. 
\end{theorem}

\begin{definition}\label{def0.7}
Let $p\in\bn$, we say that {\it spectral implies tile for period $p$} if every bounded Borel subset $\Omega$ of $\br$, with Lebesgue measure $|\Omega|=1$ and which has a spectrum $\Lambda$ of period $p$, is also a tile. 
\end{definition}

In the original paper \cite{Fug74}, Fuglede proved that his conjecture is true in the case when the spectrum or the tiling set is a lattice. This corresponds to the case of period equal to 1. We prove that spectral implies tile for periods 2,3,4,5.

\begin{theorem}\label{th0.9}
Spectral implies tile for period $2$, $3$, $4$, or $5$. 
\end{theorem}

We end the paper with some examples to illustrate our results. Example \ref{ex4.1} shows that in the well known Jorgensen Pedersen example, of a scale 4 Cantor set, the spectrum $\Lambda$ described in \eqref{eqspmu4} tiles $\bz$ with translations and the tiling set is the spectrum of a complementary fractal measure.

\section{Proofs and other results}

\subsection{Spectra of fractals and base $R$ representations of integers}

\begin{proposition}\label{pr1.2}
Let $d$ be the greatest common divisor of the points in $B$. Let $M=\max\{ l: l\in L\}$, $m=\min\{ l: l\in L\}$. Then for every extreme cycle point $x$ for $(B,L)$ we have $x\in\frac1d\bz$ and $\frac m{R-1}\leq x\leq \frac{M}{R-1}$. 
\end{proposition}

\begin{proof}
Since $|m_B(x)|=1$ and $0\in B$, using the triangle inequality we obtain that $e^{2\pi i bx}=1$ for all $b\in B$. Therefore $bx\in\bz$ for all $b\in B$. 
This implies that $dx\in\bz$ so $x\in\frac1d\bz$. 

Let $x=x_0, x_1,\dots, x_{r-1}$ be a cycle for $L$, with digits $l_0,\dots, l_{r-1}$. Then we have 
$$\frac{x_0+l_0+Rl_1+\dots+R^{r-1}l_{r-1}}{R^r}=x_0\mbox{ so }x_0=\frac{l_0+Rl_1+\dots+R^{r-1}l_{r-1}}{R^r-1}$$
which implies
$$x_0\leq \frac{M\frac{R^{r}-1}{R-1}}{R^r-1}=\frac{M}{R-1},$$
and similarly for the lower bound. 
\end{proof}

\begin{proposition}\label{pr1.3}
Let $L$ be a complete set of representatives $\mod R$. Then every integer $x$ has a unique representation in base $R$ using digits in $L$. Moreover any such representation $l_0l_1\dots$ is eventually periodic, i.e., there exists $n_0\geq0$ and $r\geq 1$ such that $l_{n+r}=l_n$ for all $n\geq n_0$.
\end{proposition}

\begin{proof}
Let $x\in\bz$ and $x_0=x$. Since $L$ is a complete set of representatives $\mod R$, there is a unique $l_0\in L$ and some $x_1\in\bz$ such that $x_0=Rx_1+l_0$. By induction, we obtain the sequence $\{x_n\}$ and $\{l_n\}$ in a unique way. We have to show that the sequence $\{l_n\}$ is eventually periodic. 
Let $M=\max\{|l| : l\in L\}$. By induction we can show that, for all $n\in\bn$,
$$x_n=\frac{x_0-l_0-Rl_1-\dots-R^{n-1}l_{n-1}}{R^n}.$$
This implies that for $n$ large enough such that $|x_0/R^n|\leq 1$ we have 
$$| x_n|\leq 1+M\left(\frac{1}R+\dots+\frac{1}{R^{n-1}}\right)\leq1+ \frac M{R-1}.$$
So $x_n$ lies in a compact interval from some point on. But $x_n$ is also an integer so the numbers $x_n$ take finitely many values. Therefore there exists $n_0\geq 0$ and $r\geq 1$ such that $x_{n_0}=x_{n_0+r}$. This implies that $l_{n_0}=l_{n_0+r}$ and $x_{n_0+1}=x_{n_0+r+1}$. By induction $l_{n}=l_{n+r}$ for all $n\geq  n_0$.  
\end{proof}

\begin{definition}\label{def1.4}
If $L$ is a complete set of representatives $\mod R$, we write $x=l_0l_1\dots$ if $l_0l_1\dots$ is the base $R$ representation of $x$ using digits in $L$. For $l_0,\dots, l_{r-1}$ in $L$ we denote by $\uln{l_0\dots l_{r-1}}$ the periodic sequence $l_0\dots l_{r-1}l_0\dots l_{r-1}\dots$. If $x=\uln{l_0\dots l_{r-1}}$ for some $l_0,\dots,l_{r-1}\in L$, we say that $x$ has a {\it periodic} representation in base $R$ using digits in $L$.

We say that $l_0\dots l_{r-1}$ is a {\it cycle} for $L$ if there exists an integer that has base $R$ representation equal to $\uln{l_0\dots l_{r-1}}$.
\end{definition}

\begin{proposition}\label{pr1.5}
If $\{x_0,\dots,x_{r-1}\}$ is a cycle for $L$ with digits $l_0,\dots, l_{r-1}$, and if $x_0\in\bz$ then $x_1,\dots, x_{r-1}\in\bz$ and the points $-x_0,\dots,-x_{r-1}$ have periodic expansions in base $R$ using digits in $L$:
\begin{equation}
-x_0=\uln{l_0\dots l_{r-1}},\quad -x_1=\uln{l_1\dots l_{r-1}l_0},\quad\dots\quad, -x_{r-1}=\uln{l_{r-1}l_0\dots l_{r-2}}.
\label{eq1.5.1}
\end{equation}
Conversely, if $-x_0\in\bz$ has a periodic expansion in base $R$ using digits in $L$ , $-x_0=\uln{l_0\dots l_{r-1}}$, and we define 
$$x_1=-\,\uln{l_1\dots l_{r-1}l_0},\quad\dots\quad, x_{r-1}=-\,\uln{l_{r-1}l_0\dots l_{r-2}},$$
then $\{x_0,\dots, x_{r-1}\}$ is a cycle for $L$ contained in $\bz$. 

\end{proposition}

\begin{proof}
If $\{x_0,x_1\dots,x_{r-1}\}$ is a cycle for $L$ with digits $l_0,\dots, l_{r-1}$ then $-x_{r-1}=R(-x_0)+l_{r-1}$ so $x_{r-1}\in\bz$. By induction, all points in this cycle are in $\bz$. We have also $-x_0=R(-x_1)+l_0$, $-x_1=R(-x_2)+l_1,\dots$. This shows that $-x_0=\uln{l_0\dots,l_{r-1}}$, $-x_1=\uln{l_1\dots l_{r-1}l_0}$, etc.

For the converse, if $-x_0=\uln{l_0\dots l_{r-1}}$ then $-x_0=R(-x_1)+l_0$ and the number $-x_1$ will have the representation $\uln{l_1\dots l_{r-1}l_0}$. 
This implies also $(x_0+l_0)/R=x_1$. The rest follows by induction.
\end{proof}

%\begin{proposition}\label{pr1.6}
%Let $(B,L)$ be a Hadamard pair. Assume in addition that all cycles for $L$ which are extreme for $B$ are contained in $\bz$. Then the spectrum $\Lambda$ defined in Theorem \ref{th1.2} is the set of integers which can be represented in base $R$ using digits in $L$.
%\end{proposition}

\begin{myproof}[Proof of Proposition \ref{pr0.1.6}]
First, note that the points in $L$ are incongruent $\mod R$. Indeed if $l-l'=Rk$ for some $k\in\bz$, then from the unitarity of the matrix in Definition \ref{def1.1i} we have
$$0=\frac1N\sum_{b\in B}e^{2\pi i R^{-1}b\cdot(l-l')}=\frac{1}{N}\sum_{b\in B}e^{2\pi i b\cdot k}=1,$$
a contradiction.
From Proposition \ref{pr1.5} we see that for any extreme cycle point $x$, the point $-x$ has a periodic representation using digits in $L$. Also, if $x$ is an integer that has a periodic representation using digits in $L$, then $-x$ is an cycle point in $\bz$. Since $-x$ is in $\bz$ it follows that $m_B(-x)=1$ so $-x$ is an extreme cycle point for $(B,L)$. 

This implies that the set $\Lambda'$ of integers that can be represented in base $R$ using digits in $L$, contains $-C$ for all extreme cycles $C$. We show that $R\Lambda'+L\subset \Lambda'$. 
If $x\in\Lambda'$, $x=l_0l_1\dots$ then $Rx+l_{-1}=l_{-1}l_0l_1\dots$ so $Rx+l_{-1}\in \Lambda'$ for any $l_{-1}\in L$. 

The minimality of $\Lambda(L)$ implies that $\Lambda(L)\subset\Lambda'$. To obtain the converse inclusion, take 
$x\in\Lambda'$. With Proposition \ref{pr1.3}, $x$ has an eventually periodic expansion 
$x=k_0\dots k_{n-1}\uln{l_0\dots l_{r-1}}$. If $c=\uln{l_0\dots l_{r-1}}$ then $x=k_0+R k_1+\dots+R^{n-1} k_{n-1}+R^n c$. We have that $-c$ is an extreme cycle point so $c$ is in $\Lambda(L)$. By the invariance of $\Lambda(L)$ we get that $x\in\Lambda(L)$. So $\Lambda'\subset\Lambda$.  
\end{myproof}

%
%\begin{theorem}\label{th2.1.2}
%Let $B$ a finite set of nonnegative integers with $\gcd(B)=1$ ans which satsifies the Coven-Meyerowitz property. Let $R$ be the lowest common multiple of the elements in $S_B$. Then there exist finite sets $B',L,L'$ of nonnegatve integers such that 
%\begin{enumerate}
%	\item $(B,L)$ and $(B',L')$ are complementary Hadamard pairs (relative to the number $R$).
%	\item $B'$ satisfies the Coven-Meyerowitz property.
%\end{enumerate}
%\end{theorem}

\subsection{The Coven-Meyerowitz property and complementary Hadamard pairs}

\begin{myproof}[Proof of Theorem \ref{th0.2.1.2}]
The hard part of this theorem was covered in \cite[Theorem A]{CoMe99} where the tiling property is proved, i.e., the existence of the set $B'$ such that $B\oplus B'=\bz_R$, and in \cite[Theorem 1.5]{Lab02} where it is shown the spectral property, i.e., the existence of the set $L$. We will include parts of their proofs here to be able to get some more information.

The set $B'$ is defined as follows: first, define the polynomial $B'(x)=\prod \Phi_s(x^{t(s)})$ where the product is take over all the prime power factors of $R$ which are not in $S_A$ and $t(s)$ is the largest factor of $R$ which is prime to $s$. It is shown in \cite{CoMe99} that this polynomial has coefficients $0$ or $1$, therefore it corresponds to a set $B'$, and $B\oplus B'=\bz_R$ (addition modulo $R$). 

Take a number $s$ that appears in the product that defines $B'(x)$. Since $s$ is a prime power, say $s=p^\alpha$, the cyclotomic polynomial is of the form $\Phi_s(x)=1+x^{p^{\alpha-1}}+x^{2p^{\alpha-1}}+\dots+x^{(p-1)p^{\alpha-1}}$ (see e.g. \cite[Lemma 1.1]{CoMe99}). Hence
$$\Phi_s(x^{t(s)})=1+x^{p^{\alpha-1}t(s)}+x^{2p^{\alpha-1}t(s)}+\dots+x^{(p-1)p^{\alpha-1}t(s)}.$$

 So all the coefficients are nonnegative for all the factors that appear in this product. Therefore, there are no cancelations. This implies that $x^{p^{\alpha-1}t(s)}$ appears with a positive coefficient in $B'(x)$. So $p^{\alpha-1}t(s)$ is in $B'$. The greatest common divisor of the elements in $B'$ must divide $p^{\alpha-1}t(s)$ which divides $st(s)$, and by the definition of $t(s)$ this will divide $R$. 
 
 Therefore we have that $\gcd(B')$ divides $R$. We will use this property to show that all the extreme cycles for the Hadamard pair $(B',L')$ are in $\bz$.

 Since $B\oplus B'=\bz_R$, we have by \cite[Lemma1.3]{CoMe99} that for any prime power $s$ that divides $R$, the cyclotomic polynomial $\Phi_s(x)$ divides $B(x)$ or $B'(x)$. Then, with \cite[Lemma 2.1]{CoMe99}, we obtain that $S_B$ and $S_{B'}$ are disjoint sets whose union is the set of all prime power factors of $R$, and also 
 $B'(1)=\prod_{s\in S_{B'}}\Phi_s(1)$, so the (T1) property is satsisfied by $B'$.
 
 To see that the (T2) property is satsisfied by $B'$ we follow again the proof of \cite[Theorem A]{CoMe99}: it is shown there that if $s=s_1\dots s_m$ is a product of distinct prime power factors of $R$ and $s_i$ is not in $S_B$, then $\Phi_s(x)$ divides $\Phi_{s_i}(t(x))$ (\cite[Lemma 1.1.(6)]{CoMe99}) so it divides $B'(x)$. So, if all $s_1,\dots,s_m$ are in $S_{B'}$, then they will not be in $S_B$ so $\Phi_s(x)$ will divide $B'(x)$, which proves (T2). Hence $B'$ has the CM-property.
 
 Since $\gcd(B)=1$, we have also $\gcd(B\oplus B')=1$.

 Now we take care of the spectral part. We use the proof of \cite[Theorem 1.5]{Lab02}. The set $L$ will contain all sums of the form $R\cdot \sum_{s\in S_B}\frac{k_s}{s}$ where $s\in S_B$, $s=p^\alpha$ for some $\alpha>0$, $p$ prime and $k_s\in\{0,\dots,p-1\}$. Since $B$ satsifies the CM-property, it is shown in \cite{Lab02} that $(B,L)$ is a Hadamard pair. Obviously $L$ is a subset of $\bz$, since the elements of $S_B$ divide $R$. Similarly we can construct $L'$ for $B'$, since we showed that $B'$ has the CM-property. 
 
 Next we show that $L\oplus L'=\bz_R$. We have $|L|\cdot |L'|=|B|\cdot|B'|=R$. 
 We prove that $(L-L)\cap (L'-L')=\{0\}$ (in $\bz_R$). Suppose we have 
 \begin{equation}\label{eqdisj2.1}
R\cdot\sum_{s\in S_B}\frac{k_s-l_s}{s}=R\cdot \sum_{s\in S_{B'}}\frac{k_s-l_s}{s}\quad\mod R,
\end{equation}
where $k_s$, $l_s$ for $s$ in either $S_B$ or $S_{B'}$ are as above. We proved above that $S_B$ and $S_{B'}$ are disjoint and their union consists of all prime power factors of $R$. 
 
  Take some prime $p$ that divides $R$ and let $s=p^\alpha$ be the largest power that divides $R$. Then $s$ appears in one of the sums in \eqref{eqdisj2.1}, and $R\cdot\frac{k_s-l_s}{s}$ is not divisible by $p$ unless $k_s=l_s$. For all the other elements $s'\in S_B\cup S_{B'}$ the numbers $R\cdot\frac{k_{s'}-l_{s'}}{s'}$ are divisible by $p$. Therefore, the equality \eqref{eqdisj2.1} implies that $k_s=l_s$. By induction we assume that $k_s=l_s$ for all $s$ in $S_B$ or $S_{B'}$ of the form $s=p^\beta$ with $1\leq \gamma\leq \beta\leq\alpha$. Then consider $s=p^{\gamma-1}$, which is in either $S_B$ or $S_{B'}$. Then $R\frac{k_s-l_s}{s}$ is not divisible by $p^{\alpha-\gamma+2}$ unless $k_s=l_s$ and for the other $s'\in S_B\cup S_{B'}$ for which $k_s\neq l_s$, the number $R\frac{k_{s'}-l_{s'}}{s'}$ is divisible by $p^{\alpha-\gamma+2}$. Using equation \eqref{eqdisj2.1} we obtain that $k_s=l_s$. Therefore for all the powers $s$ of $p$ that appear in either $S_B$ or $S_{B'}$ we have $k_s=l_s$. Since the prime $p$ was an arbitrary prime factor of $R$, we get that $k_s=l_s$ for all $s\in S_B\cup S_{B'}$. Hence $(L-L)\cap(L'-L')=\{0\}$ in $\bz_R$. This means that the map from $L\times L'$ to $\bz_R$, $(l,l')\mapsto l+l'\mod R$ is injective. But since $|L\times L'|=R$ the map will be also surjective so $L\oplus L'=\bz_R$.
  
  It remains to deal with the extreme cycles. By Proposition \ref{pr1.2}, since $\gcd(B)=1$, we have that the extreme cycles for $(B,L)$ are in $\bz$. 
  We also proved above that $d':=\gcd(B')$ divides $R$. By Proposition \ref{pr1.2} any extreme cycle for $(B',L')$ is contained in $\frac1{d'}\bz$. Take the first two points in such a cycle $x_0=\frac{k_0}{d'}$, $x_1=\frac{k_1}{d'}$, and for some $l_0'\in L'$:
  $$\frac{\frac{k_0}{d'}+l_0'}{R}=\frac{k_1}{d'}.$$
  Then $$\frac{k_0}{d'}+l_0'=R\cdot\frac{k_1}{d'}.$$
  But since $R$ is divisble by $d'$, it follows that $R\cdot\frac{k_1}{d'}$ is in $\bz$; also $l_0'$ is in $\bz$, so $x_0=\frac{k_0}{d'}$ is in $\bz$. Thus all extreme cycles for $(B',L')$ are contained in $\bz$.

\end{myproof}
%
%\begin{theorem}\label{th0.2.3}
%Let $(B,L)$ and $(B',L')$ be complementary Hadamard pairs. Let $\Lambda(L)$ be the set of integers that can be represented in base $R$ using digits from $L$, and similarly for $\Lambda(L')$.
%\begin{enumerate}
%	\item The measure $\mu_{B\oplus B'}$ is the Lebesgue measure on the attractor $X_{B\oplus B'}$ and has spectrum $\bz$. Moreover $X_{B\oplus B'}$ is translation congruent to $[0,1]$, i.e., there exists a measurable partition $\{A_n\}_{n\in\bz}$ of $[0,1]$ such that $\{A_n+n\}_{n\in\bz}$ is a partition of $X_{B\oplus B'}$. 
%	\item The measure $\mu_{B\oplus B'}$ is the convolution of the measures $\mu_B$ and $\mu_{B'}$. 
%	\item The set $\Lambda(L)$ is a spectrum for $\mu_B$ and the set $\Lambda(L')$ is a spectrum for $\mu_{B'}$.
%	\item The sets $\Lambda(L)$ and $\Lambda(L')$ have disjoint differences. 	
%	\item The set $\Lambda(L)\oplus\Lambda(L')=\bz$ if and only if for any extreme cycle  $a_0\dots a_{r-1}$ for $L\oplus L'$, the sequence $p(a_0\dots a_{r-1})$ is an extreme cycle for $L$ and the sequence $p'(a_0\dots a_{r-1})$ is an extreme cycle for $L'$. 
%	The condition $\Lambda(L)\oplus\Lambda(L')=\bz$ means that $\Lambda(L)$ tiles $\bz$ by $\Lambda(L')$, and the extreme cycles are extreme relative to $B\oplus B'$, $B$, and $B'$ respectively.
%\end{enumerate}
%
%\end{theorem}

\begin{myproof}[Proof of Theorem \ref{th2.3}]
Since $L\oplus L'$ and $B\oplus B'$ are complete sets of representatives $\mod R$, they form a Hadamard pair. With Proposition \ref{pr1.2}, we have that all extreme cycles for $(B\oplus B', L\oplus L')$ are contained in $\bz$. With Propositions \ref{pr1.3} and \ref{pr0.1.6} we see that $\bz$ is the spectrum for $\mu_{B\oplus B'}$. Using the results from \cite{DJ11} we obtain that $\mu_{B\oplus B'}$ is the Lebesgue measure on its support $X_{B\oplus B'}$ and $X_{B\oplus B'}$ is translation congruent to $[0,1]$.

For (ii), note that 
$m_{B\oplus B'}=m_B\cdot m_{B'}$. According to \cite{DJ06}, the Fourier transforms of the measures are 
$$\widehat\mu_{B\oplus B'}(x)=\prod_{k=1}^\infty m_{B\oplus B'}(R^{-k}x)$$
and similarly for $\mu_B$ and $\mu_{B'}$ and the products are uniformly and absolutely convergent on compact sets. Therefore $\widehat\mu_{B\oplus B'}=\widehat\mu_B\cdot\widehat\mu_{B'}$ which implies that $\mu_{B\oplus B'}=\mu_B*\mu_{B'}$.

(iii) follows from Proposition \ref{pr0.1.6}.

For (iv), let $\lambda=l_0l_1\dots$, $\lambda'=l_0'l_1'\dots$, $\gamma=k_0k_1\dots$, $\gamma'=k_0'k_1'\dots$ such that $\lambda-\gamma=\lambda'-\gamma'$. Reducing $\mod R$, we obtain $l_0-k_0\equiv l_0'-k_0'\mod R$. This implies $l_0+k_0'\equiv l_0'+k_0\mod R$, but since $L\oplus L'$ is a complete set of representatives $\mod R$, we get $l_0+k_0'=l_0'+k_0$ so $l_0-k_0=l_0'-k_0'$. Since $L$ and $L'$ have disjoint differences, it follows that $l_0=k_0$ and $l_0'=k_0'$. Then, by induction $l_n=k_n$ and $l_n'=k_n'$ for all $n$, so $\lambda=\gamma$ and $\lambda'=\gamma'$.

For (v), take an extreme cycle for $(B\oplus B',L\oplus L')$ with digits $a_0\dots a_{r-1}$. Then $x=\uln{a_0\dots a_{r-1}}$ is a point in $\bz=\Lambda(L)\oplus \Lambda(L')$. Thus $x=l_0l_1\dots + l_0'l_1'\dots$. This implies that $a_0\equiv l_0+l_0'\mod R$. Since $L\oplus L'$ is a complete set of representatives $\mod R$, this means that $a_0=l_0+l_0'$ and $l_0=p(a_0)$, $l_0'=p'(a_0)$. By induction $l_n=p(a_n)$ and $l_n'=p'(a_n)$ for all $n$. So $l_0l_1\dots=p(\uln{a_0\dots a_{r-1}})$ and $l_0'l_1'\dots=p'(\uln{a_0\dots a_{r-1}})$, so $p(a_0\dots a_{r-1})$ contains the digits of an extreme cycle for $(B,L)$ and similarly for $p'$.

For the converse, note that $R(\Lambda(L)\oplus\Lambda(L'))+L\oplus L'\subset \Lambda(L)\oplus \Lambda(L')$. So, by Proposition \ref{pr0.1.6} and Proposition \ref{pr1.5}, it is enough to show that $\Lambda(L)\oplus\Lambda(L')$ contains all points $\uln{a_0\dots a_{r-1}}$ where $a_0\dots a_{r-1}$ are the digits of an extreme cycle for $(B\oplus B',L\oplus L')$. But the hypothesis implies that $p(\uln{a_0\dots a_{r-1}})$ represents a point in $\Lambda(L)$ and $p'(\uln{a_0\dots a_{r-1}})$ represents a point in $\Lambda(L')$. One can easily see that 
$$\uln{a_0\dots a_{r-1}}=p(\uln{a_0\dots a_{r-1}})+p'(\uln{a_0\dots a_{r-1}})$$
because the two sides are congruent $\mod R^n$ for all $n$. This implies that $\uln{a_0\dots a_{r-1}}\in\Lambda(L)\oplus\Lambda(L')$.

\end{myproof}

\begin{proposition}\label{pr3.4}
Let $R$ be an integer $R\geq 2$. Let $B,B'$ finite sets of integers such that $B\oplus B'=\{0,1,\dots,R-1\}$. Then $\mu_B*\mu_{B'}$ is the Lebesgue measure on $[0,1]$. If $\Lambda$ is an orthogonal set for $\mu_B$ and $\Lambda'$ is an orthogonal set for $\mu_{B'}$ then $\Lambda$ and $\Lambda'$ have disjoint differences.
\end{proposition}

\begin{proof}
The proof that $\mu_B*\mu_{B'}$ is the Lebesgue measure on $[0,1]$ is the same as the proof of Theorem \ref{th2.3}(i) and (ii), the attractor of the IFS associated to $B\oplus B'=\{0,\dots,R-1\}$ is $[0,1]$. 
Take $\lambda\neq\gamma$ in $\Lambda$ and $\lambda'\neq\gamma'$ in $\Lambda'$ such that $\lambda-\gamma=\lambda'-\gamma'$. Since $e_\lambda$ is orthogonal to $e_\gamma$, we have that $\widehat\mu_B(\lambda-\gamma)=0$. Also $\widehat\mu_{B'}(\lambda'-\gamma')=0$. But $\widehat\mu_B$ and $\widehat\mu_{B'}$ can be extended to entire functions 
$$\widehat\mu_B(z)=\int e^{-2\pi itz}\,\mu_B(t),\quad(z\in\bc)$$
and similarly for $\mu_B'$. Their product is the Fourier transform of the Lebesgue measure on $[0,1]$ which is 
$$\widehat\mu_{B\oplus B'}(z)=e^{-\pi i z}\frac{\sin\pi z}{\pi z}.$$
The zeros of $\widehat\mu_{B\oplus B'}$ on $\br$ have multiplicity one. But the relations above shows that $\lambda-\gamma=\lambda'-\gamma'$ is zero of multiplicity at least 2 for $\widehat\mu_B\cdot\widehat\mu_{B'}=\widehat\mu_{B\oplus B'}$. This contradiction proves the conclusion.
\end{proof}

\subsection{Finite sets of size 2,3,4,5}

\begin{remark}\label{remf1}
We will often ignore the multiplicative constant $\frac{1}{\sqrt{N}}$ for Hadamard matrices. So, when we say that some number $z$ with $|z|=1$ is an entry in a Hadamard matrix, we actually mean that $\frac{1}{\sqrt{N}}z$ is.

We also note here that many times the study of a spectral set $B$ in $\bz$ with spectrum $\Lambda$ in $\br$ can be reduced to the study of Hadamard pairs, so we can assume in addition that $\Lambda$ has the form $\Lambda=\frac{1}{R}L$ for some $R$ integer and $L$ in $\bz$. First, we examine what happens if $\Lambda$ has only rational numbers.

If $\beta$ is a finite subset of the rational numbers, and $\Lambda$ is a spectrum of rational numbers for $\beta$, then $B,L$ is a Hadamard pair with scaling factor $RQ$, where $R$ is the least common multiple of the denominators of the numbers in $\Lambda$ and $Q$ the least common multiple of the the denominators of the numbers in $\beta$, and $L=R \Lambda$, $B=Q \beta$.

Indeed, the matrix associated with $(\beta,\Lambda)$ is unitary, and therefore so is the matrix associated with $(Q \beta, R \Lambda)$ with scaling factor $QR$.

Now assume $\beta$ is a finite subset of the rational numbers and $\Lambda$ is a spectrum of real numbers for $\beta$. If the unitary matrix associated with $(\beta, \Lambda)$ has at least one column which contains only roots of unity, then $\Lambda$ must contain only rational numbers, because the entries of that column are $e^{2 \pi i b \lambda_j}$ for all $\lambda_j \in \Lambda$. Thus, whenever we know such a thing about the columns of the Hadamard matrices for a certain size $N=\#B$, we know from the above theorem that when considering spectra (for finite sets of integers), it is sufficient to consider Hadamard pairs. For example, this property holds true of $N=2$, $3$, $4$, and $5$. 

For the remainder of the section we assume Hadamard pairs $(B,L)$ are such that $B$ and $L$ each contain $0$ as their first element, and due to the above notions we restrict our attention to Hadamard pairs which are subsets of $\bz$.

It is clear that the first row and first column of a Hadamard matrix associated to such a Hadamard pair must contain only $1$'s (ignoring the multiplicative constant $\frac1{\sqrt N}$). Therefore, when we consider Hadamard matrices in this section, we consider only the ones which are in "de-phased" form, i.e. their first row and column contains only $1$'s. For any Hadamard matrix $H$ there are diagonal matrices $D_1$ and $D_2$ so that $D_1 H D_2$ is de-phased (see e.g. \cite{TaZy06}), so we lose no generality in dealing with matrices in this way. In addition, we only consider one ordering of the rows and the columns of a Hadamard matrix, for changing the ordering of the rows and the columns corresponds to changing the ordering of the elements in $B$ and $L$. Since we know the equivalence classes of the Hadamard matrices for $N=2$, $3$, $4$, and $5$, and we know by Corollary \ref{perm} that everything in those equivalence classes are permutation equivalent, we lose no generality in dealing with Hadamard matrices in this way.

\end{remark}

%\begin{corollary} \label{perm}
%Let $N=2$, $3$, $4$, or $5$. Any two Hadamard matrices $A$ and $B$ of size $N$ in de-phased form which are equivalent are also equivalent via permutations only, that is, there are permutation matrices $P_1$ and $P_2$ such that $A=P_1 B P_2$.
%
%\end{corollary}

%\begin{theorem} \label{standard}
%Let $B \subset \bz$ have $N$ elements and spectrum $\Lambda$. Assume $0$ is in $B$ and $\Lambda$. Suppose the Hadamard matrix associated to $B,\Lambda$ is equivalent to the standard $N$ by $N$ Hadamard matrix. Then $B$ has the form $B=d B_0$ where $d$ is an integer and $B$ is a complete set of residues modulo $N$.
%
%\end{theorem}

\begin{myproof}[Proof of Theorem \ref{standard}]

First, we need some lemmas.

\begin{lemma}
Let $H$, $H'$ be two equivalent Hadamard matrices whose first rows and columns are constant $\frac{1}{\sqrt{N}}$. Then there exist permutations $\pi$, $\psi$ of $\{1,...,N\}$ such that
\beq
{H'}_{j,k} = \frac{H_{\pi (1) \psi (1)} H_{\pi (j) \psi (k)} }{\sqrt{N} H_{\pi (j) \psi (1)} H_{\pi (1) \psi (k)} } .
\eeq
\end{lemma}

\begin{proof}
Since $H$ and $H'$ are equivalent, there are permutations $\pi$ and $\psi$ and constants $c_1,...,c_N$, $d_1,...,d_N \in \mathbb{T}$ (the unit circle) such that
\beq
H'_{j,k} = c_j d_k H_{\pi (j) \psi (k)} .
\eeq
Since $H'_{j,1} = \frac{1}{\sqrt{N}} = c_j d_1 H_{\pi (j) \psi (1)} $, we obtain $c_j = \frac{1}{\sqrt{N} d_1 H_{\pi (j) \psi (1)} } $. Similarly, $d_k = \frac{1}{\sqrt{N} c_1 H_{\pi (1) \psi (k)} } $. Since $H'_{1,1} = \frac{1}{\sqrt{N}} = c_1 d_1 H_{\pi (1) \psi (1)}$, we obtain
\beq
{H'}_{j,k} = \frac{ H_{\pi (j) \psi (k)} }{{N} c_1 d_1 H_{\pi (j) \psi (1)} H_{\pi (1) \psi (k)} } ,
\eeq
and the result follows.

\end{proof}

\begin{lemma}\label{lem2.8}
Let $H$ be a Hadamard matrix whose first row and columns are constant $\frac{1}{\sqrt{N}}$. Suppose $H$ is equivalent to the standard Hadamard matrix of size $N$. Then this matrix is permutation equivalent to the standard Hadamard matrix.

\end{lemma}

\begin{proof}
 Using the previous lemma, we find permutations $\tau,\psi$ of $\{0,1,2,\dots,N-1\}$ such that
\beq
H_{j,k} = \frac{e^\frac{2 \pi i \tau (j) \psi (k) }{N} e^\frac{2 \pi i \tau (1) \psi (1) }{N} }{\sqrt{N} e^\frac{2 \pi i \tau (1) \psi (k) }{N} e^\frac{2 \pi i \tau (j) \psi (1) }{N} } = \frac{1}{\sqrt{N}} e^\frac{2 \pi i ( \tau (j) - \tau (1) )( \psi (k) - \psi (1) }{N} .
\eeq
Now notice that modulo $N$, the functions $\tau'(j)= \tau(j)-\tau(1)$ and $\psi'(k)=\psi(k)-\psi(1)$ are permutations of $\{0,1,2,\dots,N-1\}$. Thus $H$ is permutation equivalent to the standard Hadamard matrix.

\end{proof}

Now assume that $B \subset \bz$ with spectrum $\Lambda$ has $N$ elements, and $0$ is in both sets, and the matrix associated with $B$ and $\Lambda$ is equivalent to the standard Hadamard matrix of size $N$. If the greatest common divisor $d$ of $B$ is $1$, we may perform our calculations on the sets $\frac{1}{d} B$ and $d \Lambda$, which have the same associated matrix. Therefore, we assume without loss of generality that the greatest common divisor of $B$ is $1$. 

We apply the lemma above, and relabel the elements in $B$, so that $C$ is a permutation of $B$ and $\Gamma$ is a permutation of $\Lambda$, with elements $c_0=0,c_1,\dots,c_{N-1}$ and $\gamma_0=0,\gamma_1,\dots,\gamma_{N-1}$ respectively, and the matrix associated to $C$ and $\Gamma$ is the standard Hadamard matrix of size $N$. From the second row of this matrix, we obtain (here $i$ is the complex number, not an index)
\beq
e^{2 \pi i  c_j \gamma_1 } = e^{ 2 \pi i  j/N  } ,
\eeq
so $c_j \gamma_1 = \frac{j}{N} + m_j$ for some integers $m_j$. Then we write $\gamma_1 = \frac{z_1}{z_2}$ in lowest terms, as it is a rational number. Thus $c_j \frac{z_1}{z_2} = \frac{j+N m_j}{N} $. Taking $j=1$, we find that $z_2$ is divisible by $N$, so we let $z_2 = N z_3$. Thus $c_j z_1 = (j+N m_j) z_3$. Thus, since $z_1$ and $z_3$ are mutually prime, $z_3$ divides $c_j$ for all $j$. Since we know the greatest common divisor of $C$ is one, $z_3 = 1$. Thus $c_j z_1 \equiv j$ modulo $N$, so $C$ is a complete set of residues modulo $N$. Therefore, so is $B$.

To prove that we can take $\gcd(B_0)=1$, suppose $\gcd(B_0)=e$. Then $\frac{1}{e}B_0$ is again a complete set of representatives modulo $N$; indeed, if $\frac{b_1}{e}\equiv\frac{b_2}{e}\mod N$ then $b_1\equiv b_2\mod N$ so $b_1=b_2$. Also $\gcd(B_0)=1$ and we can write $B=dB_0=de\frac{1}{e}B_0$. 

Now we consider $\Lambda$. Examining the second column of the standard matrix, we find that $e^{2 \pi i c_1 \gamma_k} = e^{2 \pi i k/N}$. Therefore $\gamma_k$ is rational for all $k$, so $\Lambda$ is a set of rational numbers. Let $R$ be their lowest common denominator. Then $\Lambda = \frac{1}{R} L$ where $L$ is a set of integers containing zero. Thus $L$ is spectral with spectrum $\frac{1}{R} B$. So $L = f L_0$ where $L_0$ is a complete set of residues modulo $N$ with greatest common divisor one.

We now have that $(B,L)$ is a Hadamard pair with scaling factor $R$, whose matrix $H$ is equivalent (and thus permutation equivalent) to the standard Hadamard matrix. We assume without loss of generality that $H$ is the standard Hadamard matrix (after changing the order of the elements in $B$ and $L$). We let the elements in $B_0$ and $L_0$ be $b_j$ and $l_k$ respectively, $b_0=l_0=0$. Then
\beq
e^{\frac{2 \pi i d f b_j l_k}{R}} = e^{\frac{2 \pi i j k}{N}}.
\eeq
Thus there are integers $m_{j,k}$ such that
\beq
Ndfb_j l_k = R(jk+ Nm_{j,k}) .
\eeq
Letting $j=k=1$, we have that $N$ divides $R$ and thus $R=NS$. Thus
\beq
dfb_j l_k = S(jk+ Nm_{j,k}) .
\eeq
Thus $S$ divides $dfW$ where $W$ is the product of the greatest common divisors of $B_0$ and $L_0$, and thus $W=1$. Therefore $S$ divides $df$, so $df=St$. Thus
\beq
tb_j l_k = jk+ Nm_{j,k} .
\eeq
Thus $tb_1 l_1 = 1+ Nm_{j,k}$ so $t=\frac{df}{S}$ is mutually prime with $N$.

Conversely, let $B= d B_0$, $L= f L_0$ and $R = NS$ where $S$ divides $df$ and that quotient $t$ is mutually prime with $N$. Assume $B_0$ and $L_0$ are complete sets of residues modulo $N$. Since $t$ is mutually prime with $N$, $t B_0$ is a complete set of residues modulo $N$. Reorder $t B_0$ and $L_0$ from least to greatest modulo $N$. Then the matrix associated with $B$ and $L$ with scaling factor $R$ has entries
\beq
e^{\frac{2 \pi i d f b_j l_k}{R} } = e^{\frac{2 \pi i t b_j k}{N} } = e^{\frac{2 \pi i  j k}{N} }.
\eeq
Thus the matrix associated with $B,L$ with scaling factor $R$ is equivalent to the standard Hadamard matrix of size $N$, so $B,L$ is a Hadamard pair with scaling factor $R$. The same reasoning applies to any spectrum of rational numbers which meets the criteria.

\end{myproof}

The above classifies the Hadamard pairs for a certain class which contains all Hadamard pairs of size $N=2$, $3$, and $5$. More specifically, we have the next item.
%
%\begin{theorem}\label{pr0.0.1}A set $B \subset \bz$ with $|B|=N=2$, $3$, or $5$, where $0 \in B$ is spectral if and only if $B= N^k B_0$ where $k$ is a positive integer and $B_0$ is a complete set of residues modulo $N$.
%\end{theorem}

\begin{myproof}[Proof of Corollary \ref{pr0.1}]
All Hadamard matrices of size $2$, $3$ and $5$ are equivalent to the respective standard Hadamard matrices of those sizes, so by Theorem \ref{standard} the spectral sets are in the form $B=H_0 B_0$. We know $B_0$ contains $0$, and that $B_0$ is a complete set of residues modulo $N$. We let $H_0 = q N^k$ with $q$ not divisible by $N$. Then, since $N$ is prime in this special case, $q$ is an automorphism of the integers modulo $N$, so $N^k q B_0 = N^k B_1$ where $B_1=qB_0$ is a complete set of residues modulo $N$ which contains $0$.

\end{myproof}

We now move on to $N=4$, a case where there are other types of Hadamard matrices.

\begin{myproof}[Proof of Corollary \ref{perm}]
For $N$ equal to $2$, $3$, or $5$, all dephased Hadamard matrices are equivalent to the standard one, so by Lemma \ref{lem2.8} they are permutation equivalent to it.

Let $N=4$. Let $A$ and $B$ be equivalent de-phased Hadamard matrices, where
$$A=\begin{pmatrix}
  1&1&1&1\\
  1&-1&q&-q\\
  1&-1&-q&q\\
  1&1&-1&-1
\end{pmatrix}.
$$
We shall prove that $B$ is permutation equivalent to $A$.
Before we proceed, let us prove a lemma:
\begin{lemma}\label{lem3i.1}
If the numbers $\alpha,\beta,\gamma\in\bt=\{z\in\bc : |z|=1\}$ satisfy the relation
\begin{equation}
1+\alpha+\beta+\gamma=0,
\label{eq3i.1.1}
\end{equation}
then one of them must be $-1$.
\end{lemma}

\begin{proof}
Take the conjugate in \eqref{eq3i.1.1} and multiply by $\alpha\beta\gamma$: $\alpha\beta\gamma+\alpha\beta+\alpha\gamma+\beta\gamma=0$. Multiply \eqref{eq3i.1.1} by $\alpha\beta$: $\alpha\beta+\alpha^2\beta+\alpha\beta^2+\alpha\beta\gamma=0$. Now substract these two reltations to obtain
$0=\alpha\beta^2-\beta\gamma+\alpha^2\beta-\alpha\gamma=(\beta+\alpha)(\alpha\beta-\gamma)$ so
$\alpha+\beta=0$ or $\alpha\beta=\gamma$. Similarly, by symmetry, we obtain $\alpha+\gamma=0$ or $\alpha\gamma=\beta$. Also $\beta+\gamma=0$ or $\beta\gamma=\alpha$.

If $\alpha+\beta=0$ then, using \eqref{eq3i.1.1}, we get that $\gamma=-1$. Therefore, if one of these relations is true, then the lemma is proved. If none is true, then we must have $\alpha\beta=\gamma$ and $\alpha\gamma=\beta$ and $\beta\gamma=\alpha$. Multiply them: $\alpha^2\beta^2\gamma^2=\alpha\beta\gamma$, so $\alpha\beta\gamma=1$. Multiply the first relation by $\gamma$, we obtain that $\gamma^2=1$. Similarly $\alpha^2=1$ and $\beta^2=1$. So $\alpha,\beta,\gamma=\pm 1$. We cannot have all of them $1$, because of \eqref{eq3i.1.1}, therefore one has to be -1.
\end{proof}

Therefore the matrix $B$ has the number negative one in every row and every column. 
Thus each row and column has a 1 and -1. Consider now the other entries of the matrix which are not $\pm 1$. If we fix one, denote it by $t$ then the other non $\pm1$ entries which lie on the same row or column will have to be $-t$, because of \eqref{eq3i.1.1}. Using the same procedure we can fill out some more entries by $t$ and all the entries of the matrix are completely determined in this way. Now suppose we have two rows such that the entries 1 and $-1$ do not match, for example $(1, -1, \ast,\ast)$ and $(1,\ast, -1,\ast)$. Then
the two rows will be of the form $(1,-1, t,-t)$ and $(1,-t,-1,t)$. By orthogonality, we get $0=1+\cj t-t-t\cj t=\cj t-t$. So $t$ has to be real so $t=\pm1$.

If we have two rows such that the $\pm1$ entries match, for example $(1,-1,t,-t)$ and $(1,-1, -t, t)$, then the last row is forced to be $(1,1,-1,-1)$.
Thus, in both cases, one of the rows has to have two ones and two $-1$. Similarly, one of the columns has the same property. Thus $B$ is permutation equivalent to a matrix of the form:

$$C=\begin{pmatrix}
  1&1&1&1\\
  1&-1&t&-t\\
  1&-1&-t&t\\
  1&1&-1&-1
\end{pmatrix}.
$$
Therefore $A$ is equivalent to $C$. Now let us prove another lemma. This lemma is found in \cite{TaZy06}, where it is attributed to Haagerup \cite{Haa97}, though it does not appear in its present form in \cite{Haa97}.

\begin{lemma}\label{haag}
Let $H$ be a Hadamard matrix and consider the set $T(H) = \{ H_{j,k} \overline{H_{n,k}} H_{n,m} \overline{H_{j,m}} \}$. If $A$ and $B$ are equivalent Hadamard matrices, $T(A) = T(B)$. The set $T$ is called the invariants of a Hadamard matrix.

\end{lemma}

\begin{proof}
Assume $A$ and $B$ are permutation equivalent. Then, since they have the same entries, $T(A) = T(B)$. Then it is sufficient to prove that if $A$ and $C$ are equivalent via diagonal matrices $X$ and $Y$, so that $A = X C Y$, then they have the same invariants. Note that $A_{j,k} = X_{j,j} C_{j,k} Y_{k,k}$. We compute the elements of $T(A)=T(XCY)$:
\beq
A_{j,k} \overline{A_{n,k}} A_{n,m} \overline{A_{j,m}} = X_{j,j} C_{j,k} Y_{k,k} \overline{X_{n,n} C_{n,k} Y_{k,k}} X_{n,n} C_{n,m} Y_{m,m} \overline{X_{j,j} C_{j,m} Y_{m,m}}
\eeq
Simplifying the right hand side, we then obtain the desired result as $X_{q,q} \overline{X_{q,q}} =1$.

\end{proof}

Now notice that the above lemma implies that if two Hadamard matrices are equivalent and de-phased, they have the same elements (we let $j$ and $k$ be any numbers and the rest of the indices be $1$). Therefore, examining $A$ and $C$, we can see that $t = \pm q$, so $A$ and $C$ are permutation equivalent. But $B$ and $C$ are permutation equivalent, so $A$ and $B$ are permutation equivalent.

\end{myproof}
\begin{myproof}[Proof of Theorem \ref{thha4}]

We first prove that the spectra can be decomposed as $\frac1R L$. We know from Theorem \ref{th1.15} and Corollary \ref{perm} that the matrix associated with $B$ and $\Lambda$ is permutation equivalent with a matrix that has a $-1$ in every row except the first. Therefore for each non-zero element $\lambda_k$ of $\Lambda$ there is a $j$ such that $b_j \in B$ and $e^{2 \pi i b_j \lambda_k} = -1$. Therefore since $b_j$ are integers, $\Lambda$ is a set of rational numbers. So we let $\Lambda = \frac{1}{R} L$ where $L$ is a set of integers containing $0$, and we have the result.

 Using Theorem \ref{th1.15} and Corollary \ref{perm} we have that the matrix $H:=\frac1{\sqrt{4}}\left(e^{2\pi ib\lambda}\right)_{b\in B,\lambda\in\Lambda}$ is of the form given in \eqref{eqmat4}, after some pemutations of $B$ and $\Lambda$. This means, upon some relabelling, that we have for some $\lambda\in\Lambda$ and $B=\{0,b_1,b_2,b_3\}$: $e^{2\pi ib_1\lambda}=1$, $e^{2\pi i b_2\lambda}=-1$, $e^{2\pi ib_3\lambda}=-1$. Therefore $b_1\lambda=k_1, b_2\lambda=\frac{2k_2+1}2,b_3\lambda=\frac{2k_3+1}2$ for some $k_1,k_2,k_3\in\bz$.

We can write $b_1=2^{a_1}c_1$, $b_2=2^{a_2}c_2$, $b_3=2^{a_3}c_3$ with $a_1,a_2,a_3\geq0$ in $\bz$ and $c_1,c_2,c_3$ odd. We get that $\frac{2^{a_1}c_1}{2^{a_2}c_2}=\frac{2k_1}{2k_2+1}$ so $2^{a_1}c_1(2k_2+1)=2^{a_2+1}k_1c_2$. This implies that $a_1\geq a_2+1$.

Also $\frac{2^{a_2}c_2}{2^{a_3}c_3}=\frac{2k_2+1}{2k_3+1}$ so $2^{a_2}c_2(2k_3+1)=2^{a_3}c_3(2k_2+1)$, which implies that $a_2=a_3$.

Since $B$ is spectral iff $\frac{1}{2^{a_2}}B$ is spectral, we can assume, without loss of generality, dividing by $\frac{1}{2^{a_1}}$, that $B$ is of the form
$$B=\{0,2^ac_1,c_2,c_3\},$$
with $a\geq 1$, $c_1,c_2,c_3$ odd.

Since every row has a $-1$, there is a $\lambda_2\in\Lambda$ such that $e^{2\pi i 2^a c_1\lambda_2}=-1$. Therefore $2^{a+1}c_1\lambda_2=2m+1$ for some $m\in\bz$. So $\lambda_2=\frac{2m+1}{2^{a+1}c_1}$. The other two entries on the column of $\lambda_2$ must be opposite:
$$e^{2\pi i c_2\frac{2m+1}{2^{a+1}c_1}}=-e^{2\pi i c_3\frac{2m+1}{2^{a+1}c_1}},$$
which means that
$$\frac{2m+1}{2^ac_1}c_2=\frac{2m+1}{2^ac_1}c_3+2q+1,$$
for some $q\in\bz$. Then $(2m+1)c_2=(2m+1)c_3+2^ac_1(2q+1)$. This implies that $c_3-c_2=2^ad$ for some odd number $d$. This proves that $B$ has the given form.

We have proved that $B$ (containing $0$) is a set of integers with $N=4$ elements is spectral if and only if it is of the form given in the theorem.

Assume now $(B,L)$ is a Hadamard pair with scaling factor $R$. Then, since $B$ and $L$ are both spectral sets of integers, we must have  $B=2^C \{0, 2^a c_1, c_2, c_2 + 2^a c_3\}$ and $L=2^M \{0, n_1, n_1 + 2^K n_2, 2^K n_3\}$, where $c_i$ and $n_i$ are all odd, $a$ and $K$ are positive integers, and $C$ and $M$ are non-negative integers.

Recall the Hadamard matrix for $N=4$,
$$\begin{pmatrix}
  1&1&1&1\\
  1&-1&e^{\pi i q}&-e^{\pi i q}\\
  1&-1&-e^{\pi i q}&e^{\pi i q}\\
  1&1&-1&-1
\end{pmatrix}.
$$
Here $q$ is any rational number, though we will see that not all rational numbers correspond to a Hadamard pair. We do not yet know which elements (other than $0$) in $B$ and $L$ are associated with which rows and columns. 

First we shall prove that the odd elements in $\{0, 2^a c_1, c_2, c_2 + 2^a c_3\}$ can not be associated with the entry $+1$ in the matrix above. Let us assume for contradiction's sake that the elements $2^C g \in B$ and $2^M f \in L$ are associated with the matrix entry $1$, where $g$ is odd. Then $\text{exp} \left( \frac{2 \pi i 2^{C+M} g f}{R} \right) = 1$, so $\frac{2 \pi i 2^{C+M} g f}{R} = 2 \pi  i Z $ for some integer $Z$. Then, the matrix entry associated with $2^{C+a} c_1 \in B$ and  $2^M f \in L$ must be $-1$, as $0$ is associated with $1$ and $-1$ are the only other entries in that column. Then $\text{exp} \left( 2 \pi  i \frac{2^{a+M+C} c_1 f}{R} \right) = -1 = \text{exp} \left( 2 \pi  i \frac{2^{a+M+C} c_1 f g}{R g} \right)$. Substituting, $-1 = \text{exp} \left( 2 \pi  i \frac{2^{a} c_1 Z}{ g} \right)$. Since $g$ is odd, this is impossible. Therefore, the first non-zero element of $B$ (in our current ordering) must be associated with the matrix element $1$ that is not in the first row or column. By similar reasoning, so must the last element of $L$.

Therefore, the first non-zero element of $B$ is associated with the second column of the matrix, as depicted above, and the last element of $L$ is associated with the last row of the matrix. In making these statements we make use of the fact that changing the order of the elements in a set which is part of a Hadamard pair permutes the columns or rows of the associated matrix and vice versa, and that therefore it is sufficient to consider the order of the rows and columns of $A$ as depicted above.

Now we shall show $K=a$. We have, from the second column and last row: $\text{exp} (\pi i) = \text{exp} \left( \frac{2 \pi i 2^{C+M+a} c_1 g}{R} \right) = \text{exp} \left( \frac{2 \pi i 2^{C+M+K} n_3 f}{R} \right)$, where $g$ and $f$ are odd. Thus $R$ has a power of $2$ exactly equal to both $1+C+M+a$ and $1+C+M+K$, so $K=a$.

We now also know that $R= 2^{C+M+a+1} d$, where $d$ is odd. Let $c$ be the greatest common divisor of the $c_k$'s and $n$ that of the $n_k$'s. Examining column two in the matrix above, we can see that for every $2^Mg$ in $L$, we have $\exp\left(\frac{\pi i c_1g}{d}\right)=\pm1$. Therefore, $d$ must divide $c_1 g$. Thus, since $d$ is odd, $d$ divides $c_1 n_1$, then it divides $c_1 n_2$ and $c_1n_3$. Therefore $d$ divides $c_1 n$. Similarly, from the last row we have that $d$ divides $n_3 c$. From the third column and the last column, since the corresponding entries are equal or opposite, we get that $\text{exp} \left( 2 \pi i \frac{2^{a+C} c_3 l_j}{R} \right) = \pm 1$ for all $l_j \in L$. Therefore, since $d$ is odd, $d$ must divide $c_3 l_j$ for every $l_j \in L$, so as before $d$ must divide $c_3 n$. Similarly, comparing the second and third rows, we have that $d$ divides $n_2 c$.
Thus, we have that $B$, $L$, and $R$ are as stated.

Conversely, it is easy to check that such a $B$, $L$, and $R$ lead to the Hadamard matrix above.

\end{myproof}

This gives a complete classification of Hadamard pairs of integers when $N=4$. 

\begin{remark}
There are Hadamard matrices that do not correspond to Hadamard pairs.

Consider the case when $q=0$ in the construction above for Hadamard matrices where $N=4$, which corresponds to the matrix
$$\begin{pmatrix}
  1&1&1&1\\
  1&-1&1&-1\\
  1&-1&-1&1\\
  1&1&-1&-1
\end{pmatrix} .
$$
Assume this matrix has a Hadamard pair, so it can be written as above. Consider the matrix element associated with $c_2$ and $n_1$. We have from the proof of Theorem \ref{thha4} that $k=\frac{c_2 n_1}{2^a d}$, where $k$ is an integer (so that the matrix entry is $-1$ or $1$), but $c_2$ and $n_1$ are odd and the denominator of the right hand side is even, so no Hadamard pair of integers has this as the associated matrix. Therefore no Hadamard pair of rational numbers has this as the associated matrix, and therefore, since every column contains an $R$th root of unity for some integer $R$, no set of integers $B$ and set of real numbers $\Lambda$ has this as the associated matrix.

At this point we recall that the Hadamard matrices for $N=6$ are not completely classified. The above example suggests a question: what are the Hadamard matrices for $N=6$ that arise from Hadamard pairs? We do not yet know how to answer this question.
\end{remark}

%end stuff added by John

\subsection{Spectral sets in $\br$}

\begin{lemma}\label{lem0.8}
Let $p\in \bn$. Assume the following statement is true: for every set $\Gamma=\{\lambda_0=0,\lambda_1,\dots,\lambda_{p-1}\}$ in $\br$, which has a spectrum of the form $\frac1p A$ with $A\subset\bz$, there exists a subset $\mathcal T$ of $\bz$ such that for any spectrum of $\Gamma$ of the form $\frac1pA'$ with $A'\subset\bz$, the set $A'$ tiles $\bz$ by $\mathcal T$. 

Then spectral implies tile for period $p$. 

\end{lemma}

\begin{proof}
The result follows from \cite{DJ12}.

\end{proof}

%\begin{theorem}\label{th0.0.9}
%Spectral implies tile for period $2$, $3$, $4$, or $5$. 
%\end{theorem}

\begin{myproof}[Proof of Theorem \ref{th0.9}] We use Lemma \ref{lem0.8}.

For $p=2$, take a set $\Gamma=\{0,\lambda\}$ which has a spectrum of the form $\frac12A$ with $A\subset\bz$. Using a translation we can assume $0\in A$, so $A=\{0,b\}$ with $b\in\bz$. Write $b=2^ac$ with $a\geq 0$, $c$ odd. Since $\frac12 A$ is a spectrum for $\Gamma$, the matrix $\frac1{\sqrt2}(e^{2\pi i\lambda a})_{\lambda\in\Lambda,a\in A}$ is unitary and the first row is $\frac{1}{\sqrt2}(1,1)$ and the second is $\frac{1}{\sqrt2}(1, e^{2\pi i\lambda\frac12 2^a c})$. Therefore $e^{2\pi i\lambda\frac12 2^ac}=-1$, hence $\frac12 2^ac\lambda=\frac12+k$ for some $k\in\bz$. Thus $\lambda=\frac{1+2k}{2^ac}$. 

Now take another spectrum of the same form $\frac12A'$ with $A'=\{0,2^{a'}c'\}$. Then $\lambda=\frac{1+2k'}{2^{a'}c}$ with $k'\in\bz$. This implies that 
$2^{a'}c'(1+2k')=2^ac(1+2k)$. Since $c$ and $c'$ are odd this means that $a=a'$. So the number $a$ depends only on $\Gamma$, not on the choice of the spectrum $\frac12A$. 

If a set $A$ is of the form $\{0,2^ac\}$ with $a\geq0$, $c$ odd then $A$ tiles $\bz$ by $\mathcal T:=\{0,1,\dots, 2^a-1\}\oplus 2^{a+1}\bz$. Indeed $\{0,c\}\oplus 2\bz=\bz$ so $2^a\{0,c\}\oplus 2^{a+1}\bz=2^a\bz$ so $A\oplus 2^a\bz\oplus \{0,1,\dots, 2^a-1\}=\bz$. Since $\mathcal T$ depends only on $a$ and not on $c$, hence it depends only on $\Gamma$ and not on the choice of the spectrum $\frac12A$, it follows that the hypothesis of Lemma \ref{lem0.8} are satisfied for $p=2$ and therefore spectral implies tile for period 2.

For $p=3$, $\Gamma=\{0,\lambda_1,\lambda_2\}$ which has a spectrum  of the form $\frac13A$ with $A\subset\bz$. Again, we can assume all the spectra contain $0$. 
Then $A$ is also spectral with spectrum $\frac13\Gamma$. From Corollary \ref{pr0.1} we see that $A=3^aB$ with $a\geq0$ and $B$ a complete set of representatives modulo 3. We claim that the number $a$ depends only on $\Gamma$, not on the choice of the spectrum $\frac13A$. As we see from the proof of Corollary \ref{pr0.1}, the first row of the matrix $(e^{2\pi i\lambda b})_{\lambda\in\Gamma,b\in\frac13A}$ is $(1,1,1)$ and the other two have the entries $\{1,e^{2\pi i/3},e^{4\pi i/3}\}$. This means that there is a $b_1\in B$ such that $e^{2\pi i \lambda_1\frac13 3^ab_1}=e^{2\pi i/3}$ and $b_1\not\equiv 0\mod 3$.
Then $\frac13 3^ab_1\lambda_1=\frac13+k$ for some $k\in\bz$, so $\lambda_1=\frac{1+3k}{3^ab_1}$. 

Now take anothe spectrum $\frac13A'$ with $A'=3^{a'}B'$. We get $\lambda_1=\frac{1+3k'}{3^{a'}b_1'}$ for some $k'\in\bz$, $b_1'\in B'$, $b_1'\not\equiv0\mod 3$. Then $3^{a'}b_1'(1+3k)=3^ab_1(1+3k')$. Since $b_1',b_1$ are not divisible by 3, it follows that $a=a'$. 

A set of the form $3^aB$ where $a\geq 0$ and $B$ is a complete set of representatives modulo 3 tiles $\bz$ by $\mathcal T:=\{0,1,\dots,3^a-1\}\oplus 3^{a+1}\bz$. Indeed 
$B\oplus 3\bz=\bz$, which implies that $3^aB\oplus 3^{a+1}\bz=3^a\bz$, so $3^aB\oplus 3^{a+1}\bz\oplus\{0,1,\dots,3^a-1\}=\bz$.

Since $\mathcal T$ depends only on $\Gamma$, Lemma \ref{lem0.8} shows that spectral implies tile for period 3.

For $p=4$, $\Gamma=\{0,\lambda_1,\lambda_2, \lambda_3 \}$ with spectrum  of the form $\frac14A$ with $A\subset\bz$. We assume all the spectra contain $0$.
Then $A$ is also spectral with spectrum $\frac14\Gamma$. From Theorem \ref{thha4} we have $A=2^m \{0, 2^a c_1, c_2, c_2 +2^a c_3 \}$ where all $c_i$ are odd, $m$ and $a$ are integers, and $a$ is positive.
In the present case, up to permutation of rows and columns, all Hadamard matricies are equivalent to the following (we omit the constant $\frac{1}{2}$): %citation
$$\begin{pmatrix}
  1&1&1&1\\
  1&-1&e^{\pi i q}&-e^{\pi i q}\\
  1&-1&-e^{\pi i q}&e^{\pi i q}\\
  1&1&-1&-1
\end{pmatrix},
$$
where $q$ is a real (even rational) number.
Without loss of generality, we assume that $\lambda_1$ is associated to the first column of the matrix. Then we deduce that $2^{m+a} c_1$ is associated to the last row of the matrix (see the proof of Theorem \ref{thha4}). Hence, $e^{2 \pi i \frac14 \lambda_1 2^{m+a} c_1}=1$. Thus, from the second column and either (or both) the second or third row, $e^{2 \pi i \frac14 \lambda_1 2^{m} c_2}=-1$. Thus $\lambda_1 2^{m-1}c_2$ is odd. Thus, since $c_2$ is odd, $\lambda_1$ determines $m$. From the third column and last row, we obtain $e^{2 \pi i \frac14 \lambda_2 2^{m+a} c_1}=-1$. Then $\lambda_2 2^{m+a-1}c_1$ is odd. Thus, $\lambda_2$ determines $m+a$. This means that $\Gamma$ determines $m$ and $a$.

Therefore, in our calculations we take $A$ as above with $a$ and $m$ fixed. It remains to show the existence of a tile $\mathcal T$ for $A$ that depends only on $a$ and $m$, which will show by Lemma \ref{lem0.8} that spectral implies tile for period 4.

We shall turn our attention to the simpler problem of finding a tile dependent only on $a$ for $A_0 = \{0, 2^a c_1, c_2, c_2 +2^a c_3\}$. We consider this set, modulo $2^{a+1}$. We have representatives for $0$, $2^a$, an odd number, and $2^a$ plus that odd number. We consider $T_0 = \{0,2,4,6,\dots,2^a -2\}$. We notice that $T_0 \oplus A_0 = \bz (\mod  2^{a+1})$. Hence, $T_0 \oplus A_0 \oplus 2^{a+1} \bz = \bz$, so we have a tile for $A_0$. We notice that $2^m T_0 \oplus 2^m A_0 \oplus 2^m 2^{a+1} \bz = 2^m \bz$, so $\{0,1,\dots,2^m -1\} \oplus 2^m T_0 \oplus 2^{m+a+1} \bz \oplus A = \bz$. Therefore, $\mathcal T = \{0,1,\dots,2^m -1\} \oplus 2^m T_0 \oplus 2^{m+a+1} \bz$ is a tile for $A$ which depends only on $a$ and $m$, so spectral implies tile for period 4.

 For $p=5$, $\Gamma = \{0, \lambda_1, \lambda_2, \lambda_3, \lambda_4
 \}$, with spectrum $\frac{1}{5} B$ where $B$ is a set of integers
 containing $0$. Then $B$ is spectral with spectrum $\frac{1}{5}
 \Gamma$. Then, by Corollary \ref{pr0.1}, $B=5^a \{0,b_1,b_2,b_3,b_4\}$ where
 $\{0,b_1,b_2,b_3,b_4\}$ is a complete set of residues modulo $5$ and
 $a$ is a non-negative integer. We shall show that the number $a$
 depends only on $\Gamma$. \\
 With Lemma \ref{lem2.8} the matrix associated with $(B,\frac15\Gamma)$ is (after some relabeling of $B,\Gamma$):
 $$\begin{pmatrix}
  1&1&1&1&1\\
  1&w&w^2&w^3&w^4\\
  1&w^2&w^4&w&w^3\\
  1&w^3&w&w^4&w^2 \\
  1&w^4&w^3&w^2&w 
\end{pmatrix},
$$
where $w=e^{\frac{2 \pi i}{5}}$. Select $j$ so that $b_j \equiv 1 (\mod 5)$. Now select $k$ so that $e^{\frac{2 \pi i}{5}} = e^{\frac{2 \pi i b_j \lambda_k}5}$. Then $e^{\frac{2 \pi i}{5}} = e^{\frac{2 \pi i 5^a \lambda_k}5}$. Thus, $a$ depends only on $\Gamma$. Thus, it remains to show the existence of a tile $\mathcal T$ for $B$ that depends only on $a$. \\
Since $B_0=\{0,b_1,b_2,b_3,b_4\}$ is a complete set of residues modulo $5$, a tile for $B_0$ is $T_0 = 5\bz$. Therefore, a tile for $B$ is $\mathcal T = \{0,1,\dots,5^a -1\} \oplus 5^{a+1} \bz$. This tile depends only on $a$, so spectral implies tile for period $5$.

\end{myproof}

\subsection{Complementing Hadamard pairs}
Now we would like to find complementary Hadamard pairs whenever possible for the cases $N=2,3,4,5$ that we have been exploring.
%
%\begin{proposition} \label{prHP}
% Let $B,L,R$ and $F,G,R$ be Hadamard pairs of integers with the same scaling factor $R$, where $bg$ is a multiple of $R$ for every $g\in G$ and $b\in B$. Then $B\oplus F, L \oplus G$ is a Hadamard pair with scaling factor $R$.
% \end{proposition}
% 
\begin{myproof}[Proof of Proposition \ref{prHP}] One can check this directly, by verifying the orthogonality of the rows, but we show that we are in a particular case of a more general construction of Hadamard matrices. 

 We shall prove that the matrix associated with $B\oplus F, L \oplus G$ with scaling factor $R$ can be obtained by Di\c t\u a's construction (see e.g. \cite{TaZy06}), and is therefore a Hadamard matrix. Di\c ta's construction is a generalization of the fact that the tensor product of Hadamard matrices is a Hadamard matrix: Let $A$ be a Hadamard matrix and $\{Q_1, \dots, Q_k\}$ be (possibly different) Hadamard matrices. Let $\{E_1, E_2, \dots, E_k\}$ be unitary diagonal matrices whose first element is $1$, and where $E_1$ is the identity. Then the following is a Hadamard matrix:
 $$D = \begin{pmatrix}
  A_{1,1} E_1 Q_1 & A_{1,2} E_2 Q_2 & \dots & A_{1,k} E_k Q_k \\
  . & . & . & . \\
  A_{k,1} E_1 Q_1 & A_{k,2} E_2 Q_2 & \dots & A_{k,k} E_k Q_k \\
\end{pmatrix} .
$$
Consider one way to write the matrix elements of the tensor product of matrices of size $N$:
\beq
(A\otimes B)_{\alpha, \beta} = A_{j,l} B_{m,n},
\eeq
where $\alpha = N(j-1) +m $ and $\beta = N(l-1) +n$. As one varies $n$, $m$, $j$, and $l$, one obtains the elements of $(A\otimes B)$. We generalize this formula to fit Di\c{t}\u a's construction, and assume $A$ is size $J$ and the $Q$s and $E$s are size $N$:
\beq
D_{\alpha, \beta} = A_{j,l} \left( E_l Q_l \right)_{m,n} ,
\eeq
where $\alpha = J(j-1) +m $ and $\beta = J(l-1) +n$. As before one varies the indexes on the right to obtain the entries in $D$. We notice that the $E$s are diagonal matrices, and therefore  
\beq \label{dita}
D_{\alpha, \beta} = A_{j,l} ( E_l )_{m,m} (Q_l )_{m,n} .
\eeq

Now consider the matrix associated with $B\oplus F, L \oplus G$ with scaling factor $R$. Let $B_j \in B$, $L_l \in L$, $F_m \in F$, $G_n \in G$. Thus $j$ and $l$ range from $1$ to, say, $J$, and $n$ and $m$ range from $1$ to, say, $N$. We have
\begin{equation}
X_{\alpha, \beta} =\left( \texttt{exp}\left( \frac{2 \pi i}{R}(B_j+F_m)(L_l+G_n)\right) \right)_{j,l,m,n}.
\end{equation}
The interaction of the indexes on the left and right depends on the way we organize $B\oplus F$ and $ L \oplus G$. We shall choose to organize $B\oplus F$ in such a way that the first $N$ elements of the set are given by $B_1 + F_m$, for $1\leq m \leq N$, and so on. We shall do the same things with $L \oplus G$, fix $L$ first and vary $G$. In this way, we have determined that, as in the constructions above, $\alpha = J(j-1) +m $ and $\beta = J(l-1) +n$, and thus by varying $j$, $l$, $m$, and $n$, we obtain $X$.

From the hypothesis, $\texttt{exp}\left( \frac{2 \pi i}{R}B_j G_n\right)=1$ for $B_j\in B$, $G_n\in G$. Thus we have

\begin{equation}
X_{\alpha, \beta}=\left( \texttt{exp}\left( \frac{2 \pi i}{R}(B_j L_l)\right) \texttt{exp}\left( \frac{2 \pi i}{R}(L_l F_m)\right) \texttt{exp}\left( \frac{2 \pi i}{R}(F_m G_n)\right) \right)_{j,l,m,n}.
\end{equation}
We arrange the indices:
\beq 
X_{\alpha, \beta}=\left( \texttt{exp}\left( \frac{2 \pi i}{R}(B_j L_l)\right) \right)_{j,l} \left(\texttt{exp}\left( \frac{2 \pi i}{R}(L_l F_m)\right) \right)_{l,m} \left( \texttt{exp}\left( \frac{2 \pi i}{R}(F_m G_n)\right) \right)_{m,n}.
\eeq
This is exactly like Di\c{t}\u{a}'s construction \eqref{dita}: the role of the constants $(E_l)_{m,m}$ are played by the constants $\left(\texttt{exp}\left( \frac{2 \pi i}{R}(L_l F_m)\right) \right)_{l,m}$, and when $l$ or $m$ are $1$  this is indeed $1$, and otherwise they are roots of unity as required. In addition the matrices $\left( \texttt{exp}\left( \frac{2 \pi i}{R}(B_j L_l)\right) \right)_{j,l}$ and $\left( \texttt{exp}\left( \frac{2 \pi i}{R}(F_m G_n)\right) \right)_{m,n}$ are Hadamard matrices. Thus, the matrix associated with $B\oplus F, L \oplus G$ with scaling factor $R$ is a Hadamard matrix, so they are a Hadamard pair.
\end{myproof}

%
%\begin{theorem}
%Let $B,L$ be a Hadamard pair of integers of size $N$ prime, with scaling factor an integer $R$, where the matrix associated with $B,L$ is an $NxN$ matrix whose entries are $N$th roots of unity, and whose first column and row are only $1$s. Assume that all extreme cycles for $B,L$ are contained in $\bz$. Then $B,L$ has a complementary Hadamard pair of integers.
%\end{theorem}
\begin{myproof}[Proof of Theorem \ref{th0.4a}]
As in Theorem \ref{standard}, we have that $B=h_0 N^f \{0 = b_0, b_1 , \dots ,b_{N-1} \}$ and $L=h_1 N^g \{0 = l_0, l_1 , \dots ,l_{N-1} \}$ for some non-negative integers $f$ and $g$ and positive integers $h_0$ and $h_1$ not divisible by $N$, and $\{ b_i \}$ and $\{ l_i \}$ are complete sets of residues modulo $N$. Here we have decomposed the greatest common divisors of $B$ and $L$ into powers of $N$, and other numbers. Also $R=NS$ where $S$ divides $N^{f+g}h_0h_1$ and $Z_2:=N^{f+g}h_0h_1/S$ is prime with $N$. We then have $R= N^{f+g+1} h_0 h_1 / Z_2$, where $N$ and $Z_2$ are mutually prime.

First, we further rearrange things. Notice that since $R$ is an integer, $Z_2$ must divide $h_0 h_1$.  We can write $h_0 = w_0 z_0$ and $h_1 = w_1 z_1$, $Z_2=z_0 z_1$. Then $z_0$ and $z_1$ are mutually prime with $N$. Therefore we may rewrite $B$ in the following way: $B=w_0 z_0 N^f \{0 = b_0, b_1 , \dots ,b_{N-1} \}$, where, since $z_0$ and $N$ are mutually prime, $z_0 \{0 = b_0, b_1 , \dots ,b_{N-1} \}$ is a complete set of residues modulo $N$. Thus we let $B=w_0  N^f \{0 = B_0, B_1 , \dots ,B_{N-1} \}$, $L=w_1 N^g \{0 = L_0, L_1 , \dots ,L_{N-1} \}$, where $\{B_k \}$ and $\{L_j \}$ are complete sets of residues modulo $N$, and $R=N^{f+g+1} w_0 w_1$.

Let $B' = T_0 \oplus T_1 \oplus T_2 \oplus T_3$ and $L' = U_0 \oplus U_1 \oplus U_2 \oplus U_3$, where
\beq
T_0 = \{ 0,1,2,\dots,w_0 -1 \} ; U_0 = \{ 0,1,2,\dots,w_1 -1 \}
\eeq
\beq
T_1 = \{ 0, w_0, 2w_0, \dots, (N^f -1) w_0 \} ; U_1 = \{ 0, w_1, 2w_1, \dots, (N^g -1) w_1 \}
\eeq
\beq
T_2 =  \{ 0, w_0 N^{f+1}  , \dots, (N^g - 1) w_0 N^{f+1}  \}; U_2 = \{ 0, w_1 N^{g+1}  , \dots, (N^f -1) w_1 N^{g+1}  \}
\eeq
\beq
T_3 =  \{ 0, w_0 N^{f+g+1}  , \dots, (w_1 - 1) w_0 N^{f+g+1}  \}; U_3 = \{ 0, w_1 N^{f+g+1}  , \dots, (w_0 -1) w_1 N^{f+g+1}  \}
\eeq

We shall show that $B' , L'$ is the desired complementary Hadamard pair.

First, we show that $B \oplus B' = \bz (\mod R)$, and likewise for $L$ and $L'$. Notice that $B \oplus T_1 = w_0 (\bz (\mod N^{f+1}))$. Then, $B \oplus T_0 \oplus T_1 = \bz (\mod N^{f+1} w_0)$. Then, $B \oplus T_0  \oplus T_1 \oplus T_2 = \bz (\mod N^{f+g+1} w_0)$. Lastly, $B \oplus T_0  \oplus T_1 \oplus T_2 \oplus T_3 = \bz (\mod N^{f+g+1} w_0 w_1)$, and we are done. Similar reasoning applies to $L'$.

Now we show that $B' , L'$ are a Hadamard pair with scaling factor $R$. By performing a few cancelations, we notice that $T_0 , U_3$ is a Hadamard pair with scaling factor $R$. Similarly, so is $T_1 , U_2$. In addition, notice that $t_1 u_3$ is a multiple of $R$ for every $t_1 \in T_1 , u_3 \in U_3 $. Thus, by Proposition \ref{prHP}, $T_0 \oplus T_1 , U_3 \oplus U_2 $ is a Hadamard pair with scaling factor $R$. Similarly, so is $T_2 \oplus T_3 , U_1 \oplus U_0$. Now notice that $tu$ is a multiple of $R$ for every $t \in T_2 \oplus T_3 , u \in U_3 \oplus U_2 $. Thus, by Proposition \ref{prHP}, $B',L'$ is a Hadamard pair with scaling factor $R$.

Now we show that $\text{gcd} (B\oplus B') = 1$. If $w_0 >1$, $1\in B'$. If not, if $f=0$, $N \in B'$ and $B$ contains an element of the form $Nk+1$ so $\gcd(B\oplus B')=1$; if $f>0$ then $1\in B'$, so we are done.

Now we show that the extreme cycles for $B',L'$ are contained in $\bz$. If $f>0$, $1\in B'$, so we are done (by Proposition \ref{pr1.2}). If not, $\text{gcd}(B')$ divides $w_0 N$, so the extreme cycle points are in $\bz / w_0 N$. Consider two such points, $x$ and $y$, where $x=\frac{y+l}{R}$ for some $l\in L'$. Upon multiplying by $R$, we notice that the left hand side is an integer, as is $l$, so $y$ is an integer, and we are done.

Thus, $B',L'$ is a Hadamard pair with scaling factor $R$.

\end{myproof}

Due to the above theorem, we have a complementary Hadamard pair for every Hadamard pair when $N=2$, $3$, and $5$, whenever such a thing is possible. We turn our attention to the case $N=4$.

%\begin{theorem}
%Let $B,L$ be a Hadamard pair of size $|B|=|L|=4$, with scaling factor $R$, and assume all extreme cycles for $B,L$ are contained in $\bz$. Then $B,L$ has a complementary Hadamard pair.
%\end{theorem}

\begin{myproof}[Proof of Theorem \ref{th0.5a}]
The cases of size 2,3,5 are covered by Theorem \ref{th0.4a} so we considered the case of size 4. As above, we have that $R=2^{C+M+a+1} d$, $B=2^C \{0, 2^a c_1, c_2, c_2 + 2^a c_3\}$, and $L=2^M \{0, n_1, n_1 + 2^a n_2, 2^a n_3\}$, where $c_i$ and $n_i$ are all odd, $a$ is a positive integer, $C$ and $M$ are non-negative integers, and $d$ divides $c_1 n$, $c_3 n$, $n_2 c$, and $n_3 c$, where $c$ is the greatest common divisor of the $c_k$'s and similarly for $n$.

We begin by constructing sets $B'$ and $L'$ such that $B\oplus B' = L\oplus L' = \bz (\mod R)$. Let $B' = T_0 \oplus T_1 \oplus T_2 \oplus T_3$ and $L' = U_0 \oplus U_1 \oplus U_2 \oplus U_3$, where
\beq
T_0 = 2^{C+1} \{0,1,2,\dots2^{a-1} -1 \} ;   U_0 = 2^{M+1} \{0,1,2,\dots2^{a-1} -1 \} ;
\eeq
\beq
T_1 =\{0,1,2,\dots2^{C} -1 \} ;   U_1 =  \{0,1,2,\dots2^{M} -1 \}
\eeq
\beq
T_2 = 2^{a+C+1} \{0,1,2,\dots2^{M} -1 \} ;   U_2 = 2^{a+M+1} \{0,1,2,\dots2^{C} -1 \}
\eeq
\beq
T_3 =   U_3  = 2^{a+M+C+1} \{0,1,2,\dots, d -1 \} .
\eeq
Notice that $\{0, 2^a c_1, c_2, c_2 + 2^a c_3\} \oplus \{0,2,4,\dots,2^a -2 \} = \bz(\mod2^{a+1})$. Then $$B\oplus T_0 = 2^C \{0, 2^a c_1, c_2, c_2 + 2^a c_3\} \oplus  2^C \{0,2,4,\dots,2^a -2 \} = 2^C \bz_{2^{a+1}}.$$ Thus $B \oplus T_0 \oplus T_1 = \bz (\mod 2^{a+C+1})$. Therefore, $B \oplus B' = \bz(\mod R)$. Similar logic applies to $L$ and $L'$.

Now we must show $B',L'$ are a Hadamard pair with scaling factor $R$. Consider the polynomial
\beq
B'(z) \equiv \sum_{b' \in B'} z^{b'}  .
\eeq
Since $B'$ is a direct sum of sets, we have
\beq
B'(z) = \sum_{t_0 \in T_0} z^{t_0}  \sum_{t_1 \in T_1} z^{t_1}  \sum_{t_2 \in T_2} z^{t_2}  \sum_{t_3 \in T_3} z^{t_3} .
\eeq
Now we let $p_n (z) = \sum_{k=0}^{n-1} z^{k}$. Then, rewriting the product that is $B'(z)$, we have
\beq
B'(z) = p_{2^{a-1}} (z^{2^{C+1}}) p_{2^{C}} (z) p_{2^{M}} (z^{2^{a+C+1}}) p_{d} (z^{2^{a+M+C+1}}) .
\eeq
Now let $l_1 ' \neq l_2 ' \in L'$. We would like to show that if $q = l_1 ' - l_2 '$ then $B'\left( \text{exp} \left( \frac{2 \pi i}{R}   q \right) \right)=0$. This in turn would imply that the matrix associated with $B',L'$ and scaling factor $R$ is unitary and thus, $B',L'$ is a Hadamard pair with scaling factor $R$.

Any difference $q$ of distinct elements in $L'$ can be written
\beq \label{425}
q= q_1 + 2^{M+1} q_2 + 2^{a+M+1} q_3 + 2^{a+M+C+1} q_4 ,
\eeq
where $q_1 \in \pm \{0,1,\dots,2^M -1 \}$, $q_2 \in \pm \{0,1,\dots,2^{a-1} -1 \}$, 
$q_3 \in \pm \{0,1,\dots,2^C -1 \}$, and $q_4 \in \pm \{0,1,\dots,d -1 \}$, and at least one $q_j$ is non-zero.

Notice that since $p_n (z) (z-1) = z^n - 1$, the zeroes of $p_n$ are exactly the $n$th roots of unity other than $1$. We shall use this to prove by cases that $B'\left( \text{exp} \left( \frac{2 \pi i}{R}   q \right) \right)=0$ for any $q \in L'$.

Now assume $q \neq 0$ modulo $d$. Notice $$p_d \left( \text{exp} \left( \frac{2 \pi i}{R}   q \right)^{2^{a+C+M+1}} \right) = p_d \left( \text{exp} \left( \frac{2 \pi i}{d}   q \right) \right)=0,$$ and thus $B'\left( \text{exp} \left( \frac{2 \pi i}{R}   q \right) \right)=0.$ Thus, we may assume $q= 0$ modulo $d$, and thus we let $q = q_0 d$.

Next assume $q \neq 0$ modulo $2^M$. Then since $d$ is odd, the same is true of $q_0$. Notice $$p_{2^M} \left( \text{exp} \left( \frac{2 \pi i}{R}   q_0 d \right)^{2^{a+C+1}} \right) = p_{2^M} \left( \text{exp} \left( \frac{2 \pi i}{2^M}   q_0 \right) \right)=0,$$ and thus $B'\left( \text{exp} \left( \frac{2 \pi i}{R}   q \right) \right)=0$. Thus, we may assume $q = 0$ modulo $2^M d$, so we let $q=q_a 2^M d$. Then, from \eqref{425}, we can see that $q_1 = 0$, and thus $2^{M+1} d$ divides $q$. Thus we let $q= q_b 2^{M+1} d$.

Next assume $q_b \neq 0$ modulo $2^{a-1}$. Then  $p_{2^{a-1}} \left( \text{exp} \left( \frac{2 \pi i}{R}   q_b 2^{M+1} d \right)^{2^{C+1}} \right) = p_{2^{a-1}} \left( \text{exp} \left( \frac{2 \pi i}{2^{a-1}}   q_b \right) \right)=0$, and thus $B'\left( \text{exp} \left( \frac{2 \pi i}{R}   q \right) \right)=0$. Thus we may assume $q = 0$ modulo $2^{M+a} d$, so examining \eqref{425}, we see that $q_2=0$. Thus $2^{M+a+1} d$ divides $q$, so we let $q= q_w 2^{M+a+1} d$.

Now assume $q_w \neq 0$ modulo $2^{C}$. Then  $p_{2^{C}} \left( \text{exp} \left( \frac{2 \pi i}{R}   q_w 2^{M+a+1} d \right)  \right) = p_{2^{C}} \left( \text{exp} \left( \frac{2 \pi i}{2^{C}}   q_w \right) \right)=0$, and thus $B'\left( \text{exp} \left( \frac{2 \pi i}{R}   q \right) \right)=0$.

Thus $q$ must be a multiple of $R$, otherwise $B'\left( \text{exp} \left( \frac{2 \pi i}{R}   q \right) \right)=0$. But a difference of distinct elements in $L'$ contains no such thing, so $B'\left( \text{exp} \left( \frac{2 \pi i}{R}   q \right) \right)=0$ and thus $B',L'$ are a Hadamard pair with scaling factor $R$. 

Next we must show that the greatest common divisor of elements in $B \oplus B'$ is one. If $C>0$, this is true because $1\in T_1$. If $C=0$ then $B$ contains  an odd number and since $\gcd(T_2)$ divides $2^{a+C+1}$ we get that $\gcd(B\oplus B')=1$.  

Lastly, we must show that the extreme cycles for $B',L'$ are contained in $\bz$. By construction, the greatest common divisor of $B'$ divides $R$. Therefore all the extreme cycle points must be in $\bz / R$. Consider two such points, $\frac{x}{R}$ and $\frac{y}{R}$, consecutive in the cycle. Then we have $\frac{x}{R} = \frac{l'+\frac{y}{R}}{R}$ for some $l' \in L'$. Multiplying both sides by $R$, we can see that the left hand side is an integer. Therefore, so is the right hand side, so $\frac{y}{R}$ must be an integer. But $y$ was arbitrary, so we are done.

\end{myproof}

\section{Examples}

In the following examples, we will frequently refer to the extreme cycles of $L\oplus L'$. It is to be understood that these cycles are extreme for $(B \oplus B',L\oplus L')$ with scaling factor $R$, where, since we are dealing with complementary Hadamard pairs, the greatest common divisor of $B \oplus B'$ is $1$. We also refer to the digits of a cycle as the cycle itself. 

\begin{example}\label{ex4.1}
Let $R=4$, $B=\{0,2\}$, $B'=\{0,1\}$. Then $\mu_B$ is the 4-Cantor measure defined in \cite{JoPe98} and $\mu_{B'}$ is a contraction by 2 of this measure. Let $L=\{0,3\}$ and $L'=\{0,2\}$. We check that $(B,L)$ and $(B',L')$ are complementary Hadamard pairs. 
It is easy to check that $(B,L)$ and $(B',L')$ are Hadamard pairs and $B\oplus B'$ and $L\oplus L'$ are complete sets of representatives $\mod 4$. 
By Proposition \ref{pr1.2}, the extreme cycles for $(B,L)$ are contained in $\frac12\bz\cap [0,1]$. We can check the points $\{0,1/2,1\}$ one by one and 
we see that the extreme cycles are $\{0\}$ with digits $\underline 0$ and $\{1\}$ with digits $\uln3$.

For $(B',L')$, the extreme cycles are contained in $\bz\cap[0,2/3]$. So we have only one extreme cycle $\{0\}$ with digits $\uln0$. 

Thus, the condition (ii) in Definition \ref{def2.1} is satisfied. Condition (iii) is also satisfied. So we have complementary Hadamard pairs. 

Since $B\oplus B'=\{0,1,2,3\}$, the attractor $X_{B\oplus B'}$ is the unit interval $[0,1]$ and $\mu_{B\oplus B'}$ is the Lebesgue measure on the unit interval. Therefore the convolution of the measures $\mu_B$ and $\mu_{B'}$ is the Lebegue measure on the unit interval. 

Next, we find the extreme cycles for $L\oplus L'=\{0,2,3,5\}$. These are contained in $\bz\cap[0,5/3]$. We have $\frac{1+3}{4}=1$. So the only extreme cycles are $\{0\}$ with digits $\uln0$ and $\{1\}$ with digits $\uln3$. Since $p(3)=3$ and $p'(3)=0$ and $\uln3$ is a cycle for $L$ and $\uln0$ is a cycle for $L'$, Theorem \ref{th2.3} (v) implies that the spectrum $\Lambda(L)$ for $\mu_B$ tiles $\bz$ by the spectrum $\Lambda(L')$ for $\mu_{B'}$.

Note that $\Lambda(L)$ contains negative numbers: for example $-1$ has the representation $\uln3$, $-4$ has the representation $0\uln3$.

Take now $L=\{0,1\}$ and $L'=\{0,6\}$. One can check as above that $(B,L)$ and $(B',L')$ are complementary Hadamard pairs. The extreme cycle for $(B,L)$ is $\{0\}$ with digits $\underline 0$ and the extreme cycles for $(B',L')$ are $\{0\}$ with digits $\underline 0$ and $\{2\}$ with digits $\underline 6$. 

The spectrum $\Lambda(L)$ for $\mu_B$ is the one described in \eqref{eqspmu4}. We have $L\oplus L'=\{0,1,6,7\}$. The extreme cycles for $(B\oplus B', L\oplus L')$ are $\{0\}$ with digits $\underline 0$ and $\{ 2\}$ with digits $\underline 6$. Since $p(\underline 6)=\underline 0$ which is an extreme cycle for $(B,L)$ and $p'(\underline 6)=\underline 6$ which is an extreme cycle for $(B',L')$, it follows that $\Lambda(L)$ tiles $\bz$ with $\Lambda(L')$.
\end{example}

\begin{example}\label{ex4.2}
Let $R=4$, $B=\{0,2\}$, $B'=\{0,1\}$,  $L=\{0,1\}$ , $L'=\{0,2\}$. Then it is easy to check that $(B,L)$ and $(B',L')$ are complementary Hadamard pairs. 
The only extreme cycle for $(B,L)$ and $(B',L')$ is $\{0\}$. The spectra $\Lambda(L)$ and $\Lambda(L')$ are contained in $\bn$. 
Since $L\oplus L'=\{0,1,2,3\}$ there is a non-trivial extreme cycle for $L\oplus L'$, $1=\frac{1+3}4$. Therefore we have that $\uln3$ is an extreme cycle for $L\oplus L'$. 
But $p(\uln3)=\uln1$ and $p'(\uln3)=\uln2$ and these are not extreme cycles for $L$ and $L'$ respectively.   
\end{example}

%begin stuff by John
\begin{example}\label{ex4.3}
Let $R=6$, $B=\{0,1,2\}$, $B'=\{0,3\}$,  $L=\{0,2,10\}$ , $L'=\{0,1\}$. Then it is easy to check that $(B,L)$ and $(B',L')$ are complementary Hadamard pairs. 
The extreme cycles for $(B,L)$ are $\{0\}$ with digits $\uln0$ and $\{2\}$ with digits $\uln{(10)}$. $(B',L')$ has only the trivial cycle.

We consider $L\oplus L' = \{0,1,2,3,10,11\}$. The extreme cycles are are $\{0\}$ with digits $\uln0$ and $\{2\}$ with digits $\uln{(10)}$. It is clear that $p(\uln0)$ and $p'(\uln0)$ are cycles for $L$ and $L'$ respectively. Notice that $p(\uln{(10)})=\uln{(10)}$, which is a cycle for $L$, and $p'(\uln{(10)})=\uln0$, which is a cycle for $L'$. Therefore, by theorem \ref{th2.3}, $\Lambda(L)\oplus\Lambda(L')=\bz$.

If we replace $10$ in $L$ by $4$, we still have complementary Hadamard pairs, but the extreme cycles for $L\oplus L' = \{0,1,2,3,4,5\}$ are different, and now all the extreme cycles for $(B,L)$ and $(B',L')$ are trivial.
For $L\oplus L'$ we still have $\{0\}$ with digits $\uln0$, and now the other cycle is $\{1\}$ with digits $\uln5$. Since $p(\uln5)=\uln4$ is not a cycle for $L$ and $p'(\uln5)=\uln1$ is not a cycle for $L'$, we have by Theorem \ref{th2.3} that $\Lambda(L)\oplus\Lambda(L') \neq \bz$. 
\end{example}

\begin{example}
Let $R=8$, $B=\{0,3,4,7\}$, $B'=\{0,2\}$,  $L=\{0,3,4,7\}$ , $L'=\{0,2\}$. Then it is easy to check that $(B,L)$ and $(B',L')$ are complementary Hadamard pairs. The matrix associated with $(B,L)$ and scaling factor $R$ is interesting: it is
$$\begin{pmatrix}
  1&1&1&1\\
  1&-1&e^{\pi i /4}&-e^{\pi i /4}\\
  1&-1&-e^{\pi i /4}&e^{\pi i /4}\\
  1&1&-1&-1
\end{pmatrix}.
$$
The extreme cycles for $(B,L)$ are $\{0\}$ with digits $\uln0$ and $\{1\}$ with digits $\uln{7}$. $(B',L')$ has only the trivial cycle.
 
We consider $L\oplus L' = \{0,2,3,4,5,6,7,9\}$. The cycles are are $\{0\}$ with digits $\uln0$ and $\{1\}$ with digits $\uln{7}$. It is clear that $p(\uln0)$ and $p'(\uln0)$ are cycles for $L$ and $L'$ respectively. Notice that $p(\uln{7})=\uln{7}$, which is a cycle for $L$, and $p'(\uln{7})=\uln0$, which is a cycle for $L'$. Therefore, by Theorem \ref{th2.3}, $\Lambda(L)\oplus\Lambda(L')=\bz$.

If we replace $2$ in $L'$ by $14$, we still have complementary Hadamard pairs, but the extreme cycles for $L\oplus L' = \{0,3,4,7,14,17,18,21\}$ are different. The extreme cycles for $(B,L)$  are unchanged. The extreme cycles for $(B',L')$ are now $\{0\}$ with digits $\uln{0}$ and $\{2\}$ with digits $\uln{(14)}$.

For $L\oplus L'$ we still have $\{0\}$ with digits $\uln0$, and now we also have $\{1\}$ with digits $\uln7$, $\{2\}$ with digits $\uln{(14)}$, and $\{3\}$ with digits $\uln{(21)}$. We have $p(\uln7)=\uln7$, which is an extreme cycle for $(B,L)$, and $p'(\uln7)=\uln0$, which is an extreme cycle for $(B',L')$. We also have $p(\uln{14})=\uln0$, which is an extreme cycle for $(B,L)$, and $p'(\uln{14})=\uln{(14)}$, which is an extreme cycle for $(B',L')$. Finally, we have $p(\uln{21})=\uln7$, which is an extreme cycle for $(B,L)$, and $p'(\uln{21})=\uln{(14)}$, which is an extreme cycle for $(B',L')$. Therefore, by Theorem \ref{th2.3}, $\Lambda(L)\oplus\Lambda(L')=\bz$.

So $\Lambda(L)$ tiles with two very different tiling sets $\Lambda(\{0,2\})$ and $\Lambda(\{0,14\})$.

\end{example}

\begin{acknowledgements}
This work was partially supported by a grant from the Simons Foundation (\#228539 to Dorin Dutkay).
\end{acknowledgements}

%end stuff added by john
\bibliographystyle{alpha}	
\bibliography{eframes}

\end{document}